\documentclass[11 pt,a4paper,dvips]{article}
\usepackage{amsmath,amssymb,amsfonts,amsthm,epsfig,graphics,graphicx}

\usepackage[usenames,dvipsnames]{color}
 \def\NN{{\mathbb N}}  
 \def\RR{{\mathbb R}}  
   
 \def\ZZ{{\mathbb Z}}

\def\La{\Lambda}

\def\Wi{\widetilde}

\def\cA{{\cal A}}

\def\cF{{\cal F}}

\newtheorem{theorem}{{Theorem}}[section]
\newtheorem{proposition}[theorem]{{Proposition}}

\newtheorem{lemma}[theorem]{{Lemma}}

\newtheorem{corollary}[theorem]{{Corollary}}

\theoremstyle{definition}
\newtheorem{definition}[theorem]{{Definition}}

\theoremstyle{remark}
\newtheorem{remark}[theorem]{{Remark}}
\newtheorem{remarks}[theorem]{{Remarks}}

\title{A uniqueness theorem for transitive Anosov flows \\obtained by gluing hyperbolic plugs}
\author{Fran\c cois B\'eguin and Bin Yu\footnote{This work was partially carried during some stay of Fran\c cois B\'eguin in Shanghai and Bin Yu in Paris. We thank our universities for the financial support for these visits. Yu was partially supported by National Natural Science Foundation of China (NSFC 11871374).}}
\date{\today}

\begin{document}
\maketitle

\begin{abstract}
In a previous paper with C. Bonatti (\cite{BBY}), we have defined a general procedure to build new examples of Anosov flows in dimension 3. The procedure consists in gluing together some building blocks, called \emph{hyperbolic plugs}, along their boundary in order to obtain a closed 3-manifold endowed with a complete flow. The main theorem of~\cite{BBY} states that (under some mild hypotheses) it is possible to choose the gluing maps so the resulting flow is Anosov. The aim of the present paper is to show a uniqueness result for Anosov flows obtained by such a procedure. Roughly speaking, we show that the orbital equivalence class of these Anosov flows is insensitive to the precise choice of the gluing maps used in the construction. The proof relies on a coding procedure which we find interesting for its own sake, and follows a strategy that was introduced by T. Barbot in a particular case.
\end{abstract}

\section{Introduction}

In a previous paper written with C. Bonatti (\cite{BBY}), we have proved a result allowing to construct transitive Anosov flows in dimension 3 by ``gluing hyperbolic plugs along their boundaries". The purpose of the present paper is to study Anosov flows obtained by such a construction. We focus our attention on the diffeomorphisms that are used to glue together the boundaries of the hyperbolic plugs. We aim to understand what is the impact of the choice of these diffeomorphisms on the dynamics of the resulting Anosov flows. We will see that two gluing diffeomorphisms that are ``strongly isotopic" yield some Anosov flows that are orbitally equivalent. In other words, in~\cite{BBY}, we have proved the \emph{existence} of Anosov flows constructed by a certain gluing procedure, and the goal of the present paper is to prove a \emph{uniqueness result} for these Anosov flows.

In order to state some precise questions and results, we need to introduce some terminology.  A \emph{hyperbolic plug} is a pair $(U,X)$ where $U$ is a (not necessarily connected) compact three-dimensional manifold with boundary, and $X$ is a vector field on $U$, transverse to $\partial U$, and such that the maximal invariant set $\Lambda_X:=\bigcap_{t\in\RR} X^t(U)$ is a saddle hyperbolic set for the flow $(X^t)$.  Given such a hyperbolic plug $(U,X)$, we decompose $\partial U$ as the disjoint union of an \emph{entrance boundary} $\partial^{in} U$ (the union of the connected component of $\partial U$ where the vector field $X$ is pointing inwards $U$) and an \emph{exit boundary} $\partial^{out} U$ (the union of the connected component of $\partial U$ where the vector field $X$ is pointing outwards $U$). The stable lamination $W^s(\Lambda_X)$ of the maximal invariant set $\Lambda_X$ intersects transversally the entrance boundary $\partial^{in} U$  and is disjoint from the exit boundary $\partial^{out} U$. Hence, $L^s_X:=W^s(\La_X)\cap\partial U$ a one-dimensional lamination embedded in the surface $\partial^{in} U$. Similarly, $L^u_X:=W^u(\La_X)\cap\partial U$ a one-dimensional lamination embedded in the surface $\partial^{out} U$. We call $L^s_X$ and $L^u_X$ the \emph{entrance lamination} and the \emph{exit lamination} of the hyperbolic plug $(U,X)$. It can be proved that these laminations are quite simple: 
\begin{itemize}
\item[(i)] They contain only finitely many compact leaves.
\item[(ii)] Every half non-compact leaf is asymptotic to a compact leaf.
\item[(iii)] Each compact leaf  may be oriented  such that its holonomy is a contraction.  
\end{itemize}

Hyperbolic plugs should be thought as the basic blocks of a building game, our goal being to build some Anosov flows by gluing a collection of such basic blocks together. From a formal viewpoint, a finite collection of hyperbolic plugs can always be viewed as a single non-connected hyperbolic plug. For this reason, it is enough to consider a single hyperbolic plug $(U,X)$ and a gluing diffeomorphism $\psi:\partial^{out} U\to\partial^{in} U$. For such $(U,X)$  and $\psi$, the quotient space $M:=U/\psi$ is a closed three-manifold, and the incomplete flow $(X^t)$ on $U$ induces a complete flow $(Y^t)$ on $M$. The purpose of the paper~\cite{BBY} was to describe some sufficient conditions on $U$, $X$ and $\psi$ for $(Y^t)$ to be an Anosov flow. We will now explain these conditions. 

We say that a one-dimensional lamination $L$ is \emph{filling a surface $S$} if every connected component $C$ of $S\setminus L$ is ``a strip whose width tends to $0$ at both ends": more precisely, $C$ is simply connected, the accessible boundary of $C$ consists of two distinct non-compact leaves $\ell^-,\ell^+$ of $L$, and these two leaves $\ell^-,\ell^+$ are asymptotic to each other at both ends. We say that two laminations $L_1,L_2$ embedded in the same surface $S$ are \emph{strongly transverse} if they are transverse to each other, and moreover every connected component $C$ of $S\setminus (L_1\cup L_2)$  is a topological disc whose boundary $\partial C$ consists of exactly four arcs $\alpha_1,\alpha_2, \alpha_1',\alpha_2'$ where $\alpha_1,\alpha_1'$ are arcs of leaves of the lamination $L_1$ and $\alpha_2,\alpha_2'$ are arcs of leaves of the lamination $L_2$. We say that a hyperbolic plug $(U,X)$ has \emph{filling laminations} if the entrance lamination $L^s_X$ is filling the surface $\partial^{in} U$ and the exit lamination $L^u_X$ is filling the surface $\partial^{out} U$. Given a hyperbolic plug $(U,X)$, we say that a gluing diffeomorphism $\psi:\partial^{out} U\to\partial^{in} U$ is \emph{strongly transverse} if the laminations $L^s_X$ and $\psi_*L^u_X$ (both embedded in the surface $\partial^{in} U$) are strongly transverse. If $(U,X_1)$ and $(U,X_2)$ are two hyperbolic plugs with the same underlying manifold $U$, and $\psi_1,\psi_2:\partial^{out} U\to\partial^{in} U$ are two gluing diffeomorphisms, we say that the triples $(U,X_1,\psi_1)$ and $(U,X_2,\psi_2)$ are \emph{strongly isotopic} if one can find a continuous one-parameter family $\{(U,X_t,\psi_t)\}_{t\in [1,2]}$ so that $(U,X_t)$ is a hyperbolic plug and $\psi_t:\partial^{out} U\to\partial^{in} U$ is a strongly transverse gluing map for every $t$. The main technical result of~\cite{BBY} can be stated as follows: 

\begin{theorem}
\label{t.BBY}
Let $(U,X_0)$ be a hyperbolic plug with filling laminations such that the maximal invariant set of $(U,X_0)$ contains neither attractors nor repellers, and let $\psi_0:\partial^{out}U\to \partial^{in} U$ be a strongly transverse gluing diffeomorphism. Then there exist a hyperbolic plug $(U,X)$ with filling laminations and a strongly transverse gluing diffeomorphism $\psi:\partial^{out}U\to \partial^{in}U$ such that $(U,X_0,\psi_0)$ and $(U,X,\psi)$ are strongly isotopic, and such that the vector field $Y$ induced by $X$ on the closed manifold $M:=U/\psi$ is Anosov.
\end{theorem}

The idea of building transitive Anosov flows by gluing hyperbolic plugs goes back to~\cite{BL} where Bonatti and R. Langevin consider a very simple hyperbolic plug $(U,X)$ whose maximal invariant set is a single isolated periodic orbit and are able to find an explicit gluing diffeomorphism $\psi:\partial^{out} U\to\partial^{in} U$ so that the vector field $Y$ induced by $X$ on the closed manifold $M:=U/\psi$ generates a transitive Anosov flow. This example was later generalized by T. Barbot who defined a infinite family of transitive Anosov flows which he calls \emph{BL-flows}. These examples are obtained by considering the same very simple hyperbolic plug $(U,X)$ as Bonatti and Langevin, but more general gluing diffeomorphisms. 

Theorem~\ref{t.BBY} naturally raises the following question (see~\cite[Question 1.7]{BBY}): \emph{in the statement of Theorem~\ref{t.BBY}, is the Anosov vector field $Y$ well-defined up to orbitally equivalence}~? (recall that two Anosov flows  are said to be \emph{orbitally equivalent} if there exists a homeomorphism between their phase space mapping the oriented orbits of the first flow to the oriented orbitsthe second one). One of the main purpose of the present paper is to provide a positive answer to this question. More precisely, we will prove the following:

\begin{theorem}
\label{t.main}
Let $(U,X_1,\psi_1)$ and $(U,X_2,\psi_2)$ be two hyperbolic plugs endowed with strongly transverse gluing diffeomorphisms. Let $Y_1$ and $Y_2$ be the vector fields induced by $X_1$ and $X_2$ on the closed manifolds $M_1:=U/\psi_1$ and $M_2:=U/\psi_2$. Suppose that:
\begin{enumerate}
\item[0.] the manifolds $U$, $M_1$ and $M_2$ are orientable;
\item[1.] both $Y_1$ and $Y_2$ are transitive Anosov flows;
\item[2.] the triples $(U,X_1,\psi_1)$ and $(U,X_2,\psi_2)$ are strongly isotopic.
\end{enumerate}
Then the flows $(Y_1^t)$ and $(Y_2^t)$ are orbitally equivalent.
\end{theorem}

It should be noted that, in the statement of Theorem~\ref{t.main}, we do not require that the hyperbolic plugs $(U,X_1)$ and $(U,X_2)$ have filling laminations. So Theorem~\ref{t.main} concerns a class of Anosov flows which is larger than the class of Anosov flows provided by Theorem~\ref{t.BBY}. For example, Bonatti-Langevin's classical example and its generalizations by Barbot (BL-flows) satisfy the hypotheses of Theorem~\ref{t.main}.

\begin{figure}[ht]
\begin{center}
  \includegraphics[totalheight=5cm]{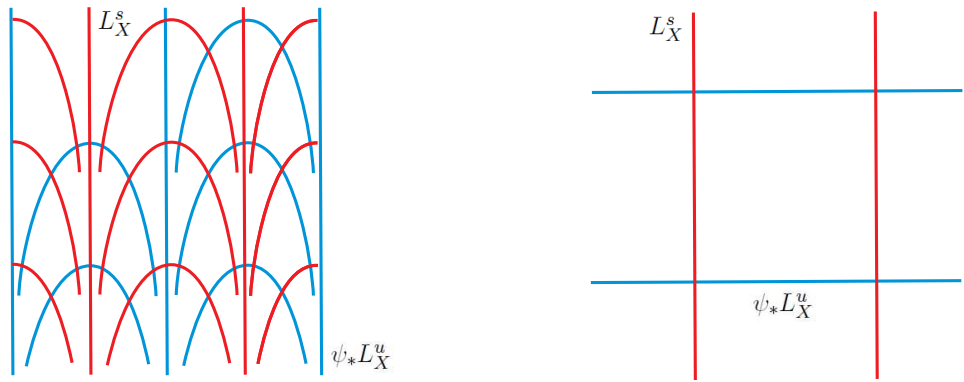}
\caption{\label{f.strongly-transverse}Two examples of strongly transverse gluing diffeomorphisms. On the left-hand side, the laminations are filling. The right-hand side corresponds to Bonatti-Langevin's example.}
 \end{center}
\end{figure}

\begin{remark}
A possible application of Theorem~\ref{t.main} is to get some finiteness results. Suppose we are given a hyperbolic plug $(U,X)$ and a diffeomorphism $\psi_0:\partial^{out} U\to\partial^{in} U$. Consider the partition of the isotopy class of $\psi_0$ into strong isotopy classes. Although we did not write down a complete proof, it seems to us that this partition is finite. In view of Theorem~\ref{t.main}, this means the following: up to orbital equivalence, there are only finitely many transitive Anosov flows that are built using the hyperbolic plug $(U,X)$ and a gluing map $\psi:\partial^{out} U\to\partial^{in} U$ isotopic to~$\psi_0$. A further consequence should be that, if we consider some given hyperbolic plugs $(U_1,X_1),\dots,(U_n,X_n)$ so that $U_1,\dots,U_n$ are hyperbolic manifolds, and if we consider a manifold $M$, then, up to orbital equivalence, there should only finitely many transitive Anosov flows on $M$ that are obtained by gluing $(U_1,X_1),\dots,(U_n,X_n)$. 
\end{remark}

An analog of Theorem~\ref{t.main} was proved by Barbot in the much more restrictive context of BL-flows (see~\cite[second assertion of Theorem~B]{Bar95}). Barbot's result can actually be considered as a particular case of Theorem~\ref{t.main}: it corresponds to the case where the maximal invariant set of the hyperbolic plug $(U_i,X_i)$ is a single isolated periodic orbit for $i=1,2$. Our proof of Theorem~\ref{t.main} roughly follows the same strategy as those of Barbot's result, but is far more intricate and requires some important new ingredients since we manipulate general hyperbolic plugs. 

The proof is based on a coding procedure that we will describe now. Consider a hyperbolic plug $(U,X)$ and a strongly transverse gluing diffeomorphism $\psi:\partial^{out} U\to\partial^{in} U$. Let $Y$ be the vector field induced by $X$ on the closed manifold $M:=U/\psi$, and assume that the flow $(Y^t)$ is a transitive Anosov flow. The projection in $M$ of $\partial U$ is a closed surface tranverse to the orbits of the Anosov flow $(Y^t)$; we denote this surface by $S$. The projection in $M$ of the entrance lamination of the plug $(U,X)$ is a lamination in the surface $S$; we denote it by $L^s$. Consider the universal cover $\widetilde M$ of the manifold $M$, and the lifts $(\widetilde Y^t),\widetilde S,\widetilde L^s$ of $(Y^t), S, L^s$. We will consider the (countable) alphabet $\cA$ whose letters are the connected components of $\widetilde S\setminus\widetilde L^s$, and the symbolic space $\Sigma$ whose elements are bi-infinite words on the alphabet $\cA$. We will construct a coding map $\chi$ from (a dense subset of) the surface $\widetilde S$ to the symbolic space $\Sigma$, commuting with the natural actions of the fundamental group of $M$, and conjugating the Poincar\'e first return map of the flow $(\widetilde Y^t)$ on the surface $\widetilde S$ to the shift map on the symbolic space $\Sigma$. If $\Lambda$ denotes the projection in $M$ of the maximal invariant set of the plug $(U,X)$, and $\widetilde\Lambda$ denotes the lift of $\Lambda$ in $\widetilde M$, then the map $\chi$ is defined at every point of $\widetilde S$ which is neither in the stable nor in the unstable lamination of $\widetilde \Lambda$. This means that the dynamics of the flow $(Y^t)$ can be decomposed into two parts: on the one hand, the orbits that converge towards to the maximal invariant set $\Lambda$ in the past or in the future, on the other hand, the dynamics that is well-described by the coding map $\chi$. 

\begin{remark}
Besides being the cornerstone of the proof of Theorem~\ref{t.main}, this coding procedure is interesting in its own sake. Indeed, it allows to understand the behavior of the recurrent orbits of the Anosov flow $(Y^t)$ that intersect the surface $S$ (\emph{i.e.} which do not correspond to recurrent orbits of the incomplete flow $(X^t)$). In a forthcoming paper~\cite{BeYu}, we will use this coding procedure to describe the free homotopy classes of theses orbits, and build new examples of transitive Anosov flows.
\end{remark}

Let us now explain how this coding procedure yields a proof of Theorem~\ref{t.main}. For $i=1,2$, we get a symbolic space $\Sigma_i$ and a coding map $\chi_i$ with value in $\Sigma_i$.  The strong isotopy between $(U,X_1,\psi_1)$ and $(U,X_2,\psi_2)$ implies that there is a natural map between the symbolic space $\Sigma_1$ and $\Sigma_2$.  Together with the coding maps, this yields a conjugacy between the Poincar\'e return maps of the flows $(\widetilde Y_1^t), (\widetilde Y_2^t)$ on the surfaces $\widetilde S_1,\widetilde S_2$. Unfortunately, this conjugacy is not well-defined on the whole surfaces $\widetilde S_1,\widetilde S_2$. So we need to extend it. In order to do that, we introduce some (partial) pre-orders on the leaf spaces of the lifts of the stable/unstable foliations of the Anosov flows $(Y_1^t),(Y_2^t)$, and prove that the  conjugacy preserves these  pre-orders. This is quite delicate since the coding maps $\chi_1,\chi_2$ do not behave very well with respect to these pre-orders. Once the extension has been achieved, we obtain a homeomorphism between the orbits spaces of the flows $(\widetilde Y_1^t)$ and $(\widetilde Y_2^t)$ that is equivariant with respect to the actions of the fundamental groups of the manifolds $M_1$ and $M_2$. Using a classical result, this implies that the Anosov flows $(Y_1^t)$ and $(Y_2^t)$ are orbitally equivalent.

\section{Coding procedure}
\label{s.coding}

In this section, we will consider a transitive Anosov flow obtained by gluing hyperbolic plugs. Our goal is to define a coding procedure for the orbits of this Anosov flow. Actually, this coding procedure will only describe the behavior of the orbits which do not remain in $\mathrm{int}(U)$ forever. 

\subsection{Setting}
\label{ss.setting}


We consider a hyperbolic plug $(U,X)$. Recall that this means that $U$ is a (not necessarily connected\footnote{Hence, a finite collection of hyperbolic plugs can always be considered as a single, non connected, hyperbolic plug.}) compact $3$-dimensional manifold with boundary, and $X$ is a vector field on $U$, transverse to $\partial U$, such that the maximal invariant set $$\Lambda_X:=\bigcap_{t\in\RR} X^t(U)$$ is a saddle hyperbolic set for the flow of $X$. We decompose the boundary of $U$ as
$$\partial U:=\partial^{in} U\sqcup \partial^{out} U$$
where $\partial^{in} U$ (resp. $\partial^{out} U$)  is the union of the connected component of $\partial U$ where $X$ is pointing inwards (resp. outwards) $U$. The stable manifold theorem implies that $W^s_X(\Lambda_X)$ and $W^u_X(\Lambda_X)$ are two-dimensional laminations transverse to $\partial U$. Moreover, $W^s_X(\Lambda_X)$ is obviously disjoint from $\partial^{out} U$ and $W^u_X(\Lambda_X)$ is obviously disjoint from $\partial^{in} U$. As a consequence,
$$L^s_X:=W^s_X(\Lambda_X)\cap \partial U=W^s_X(\Lambda_X)\cap \partial^{in} U\quad\mbox{ and }\quad L^u_X:=W^u_X(\Lambda_X)\cap \partial U=W^u_X(\Lambda_X)\cap \partial^{out} U$$
are one-dimensional laminations embedded in the surfaces $\partial^{in} U$ and $\partial^{out} U$ respectively. Note that $L^s_X$ can be described as the set of points in $\partial^{in} U$ whose forward $(X^t)$-orbit remains in $U$ forever, \emph{i.e.} does not intersect $\partial^{out} U$. Similarly,  $L^u_X$ is the set of points in $\partial^{out} U$ whose backward $(X^t)$-orbit remains in $U$ forever, \emph{i.e.} does not intersect $\partial^{in} U$. These characterizations of $L^s_X$ and $L^u_X$ allow to define a map
$$\theta_X:\partial^{in} U\setminus L^s_X\longrightarrow \partial^{out} U\setminus L^u_X$$
where $\theta_X(x)$ is the (unique) point of intersection the $(X^t)$-orbit of $x$ with the surface $\partial^{out} U$. Clearly, $\theta_X$ is a homeomorphism between $\partial^{in} U\setminus L^s_X$ and $\partial^{out} U\setminus L^u_X$. We call $\theta_X$ the \emph{crossing map} of the plug $(U,X)$. 

In order to create a closed manifold equipped with a transitive Anosov flow, we consider a diffeomorphism 
$$\psi:\partial^{out} U\to\partial^{in} U.$$  
The quotient space 
$$M:=U/\psi.$$
is a closed three-dimensional topological manifold. We denote by $\pi:U\to M$ the natural projection map. The topological manifold $M$ can equipped with a differential structure (compatible with the differential structure of $U$) so that the vector field $$Y:=\pi_* X$$
is well-defined (and as smooth as $X$). We make the following hypotheses:
\begin{enumerate}
\item[0.] the manifolds $U$ and $M$ are orientable;
\item[1.] the flow $(Y^t)$ is a transitive Anosov flow on the manifold $M$;
\item[2.] the diffeomorphism $\psi$ is a strongly transverse gluing diffeomorphism.
\end{enumerate}
 Recall that 2 means that the laminations $L^s_X$ and $\psi_*(L^u_X)$ are transverse in the surface $\partial^{in} U$ and moreover that  every connected component $C$ of $\partial^{in} U \setminus (L^s_X\cup\psi_*(L^u_X))$  is a topological disc whose boundary $\partial C$ consists of exactly four arcs $\alpha^s,{\alpha^s}', \alpha^u,{\alpha^u}'$ where $\alpha^s,{\alpha^s}'$ are arcs of leaves of $L_X^s$ and $\alpha^u,{\alpha^u}'$ are arcs of leaves  $\psi_*(L^u_X))$. We denote
$$S:=\pi(\partial^{in} U)=\pi(\partial^{out} U)\quad\quad \Lambda:=\pi(\Lambda_X)\quad\quad L^s:=\pi_*(L^s_X)\quad\quad L^u:=\pi_*(L^u_X).$$

By construction, $S$ is a closed surface, embedded in the manifold $M$, transverse to the vector field $Y$. The set $\Lambda$ is the union of the orbits of $(Y^t)$ that do not intersect the surface $S$. It is an invariant saddle hyperbolic set for the Anosov flow $(Y^t)$. Our assumptions imply that $L^s$ and $L^u$ are two strongly transverse one-dimensional laminations in the surface $S$. The lamination $L^s$ (resp. $L^u$) can be described as the sets of points in $S$ whose forward (resp. backward) $(Y^t)$-orbit does not intersect $S$. Similarly, $L^u$ is a strict subset of $W^u(\Lambda)\cap S$. The homeomorphism $\theta_X$ induces a homeomorphism
$$\theta=(\pi_{|\partial^{out}U}) \circ \theta_X\circ (\pi_{|\partial^{in}U})^{-1} : S\setminus L^s \longrightarrow S\setminus L^u.$$
Note that $\theta$ is nothing but the Poincar\'e first return map of the orbits of the Anosov flow $(Y^t)$ on the surface $S$.

Since $(Y^t)$ is an Anosov flow, it comes with a stable foliation $\cF^s$ and an unstable foliation $\cF^u$. These are two-dimensional foliations, transverse to each other, and transverse to the surface $S$. Hence, they induce two transverse one-dimensional foliations
$$F^s:=\cF^s\cap S\quad\mbox{ and }\quad F^u:=\cF^u\cap S$$
on the surface $S$. Clearly, $L^s$ and $L^u$ are sub-laminations (\emph{i.e.} union of. leaves) of the foliations $F^s$ and $F^u$ respectively.

In order to code the orbits of the Anosov flow $(Y^t)$, we cannot work directly in the manifold $M$, we need to unfold the leaves of the laminations $L^s$ and $L^u$ by lifting them to theuniversal cover of $M$. We denote this universal cover by $p:\widetilde M\longrightarrow M$, and we denote by 
$$\Wi S,\quad\Wi\Lambda,\quad \Wi W^s(\Lambda),\quad \Wi W^u(\Lambda),\quad \Wi L^s,\quad\Wi L^u, \quad\Wi\cF^s,\quad\Wi\cF^u,\quad\Wi F^s,\quad\Wi F^u,$$
the complete lifts of the surface $S$, the hyperbolic set $\Lambda$, the laminations $W^s(\Lambda), W^u(\Lambda),L^s, L^u$ and the foliations $\cF^s, \cF^u, F^s, F^u$. We insist that $\Wi S$ is the \emph{complete} lift of $S$, that is $\Wi S:=p^{-1}(S)$. In particular, $\Wi S$ has infinitely many connected components. By construction, $\Wi F^s$ and $\Wi F^u$ are two transverse one-dimensional foliations on the surface $\Wi S$, and  $\Wi L^s$ and $\Wi L^u$ are sub-laminations of $\Wi F^s$ and $\Wi F^u$ respectively. We also lift the vector field $Y$ to a vector field $\Wi Y$ on $M$. Of course, $\Wi Y$ is transverse to the surface $\Wi S$, so we can consider the Poincar\'e return map
$$\tilde\theta:\Wi S\setminus \Wi L^s\to \Wi S\setminus \Wi L^u$$
of the orbits of $(\Wi Y^t)$ on the surface $\Wi S$. Obviously, $\tilde\theta$ is a lift of the map $\theta$.

 \subsection{Connected components of $\widetilde S\setminus \widetilde L^s$}
\label{ss.connected-components}

The purpose of the present subsection is to collect some informations about the connected components of $\widetilde S\setminus \widetilde L^s$ and the action of the Poincar\'e map $\widetilde\theta$ on these connected components. These informations will be used in Subsection~\ref{ss.coding-procedure}. Let us start by the topology of the surface $\widetilde S$.

\begin{proposition}
\label{p.topology-surface}
Every connected component of $\widetilde S$ is a properly embedded topological plane.
\end{proposition}

\begin{proof}
The surface $S$ is transverse to the Anosov flow $(Y^t)$. Hence, $S$ is a collection of incompressible tori in $M$ (see \emph{e.g.} \cite[Corollary 2.2]{Fen}). The proposition follows.
\end{proof}

This allows us to describe the topology of the leaves of the foliations $\Wi F^s$ and $\Wi F^u$:

\begin{proposition}
\label{p.intersection-two-leaves}
Every leaf of the foliations $\Wi F^s$ and $\Wi F^u$ is a properly embedded topological line. A leaf of $\Wi F^s$ and a leaf of $\Wi F^u$ intersect no more than one point.
\end{proposition}

\begin{proof}
The first assertion follows immediately from Proposition~\ref{p.topology-surface}: it is a classical consequence of Poincar\'e-Hopf and Poincar\'e-Bendixon theorems that the leaves of a foliation of a plane are properly embedded topological lines.

The second assertion is again a consequence Proposition~\ref{p.topology-surface}, together with the transversality of the foliations $\Wi F^s$ and $\Wi F^u$. To prove it, we argue by contradiction: consider a leaf $\ell^s$ of $\Wi F^s$, a leaf $\ell^u$ of $\Wi F^u$, and assume that $\ell^s$ and $\ell^u$ intersect at more than one point. Then one can find two arcs $\alpha^s\subset\ell^s$ and $\alpha^u\subset\ell^u$, which  share the same endpoints and have disjoint interiors. The union $\alpha^s\cup\alpha^s$ is a simple closed curve in $\Wi S$. Since every connected component of $\Wi S$ is a topological plane, $\alpha^s\cup\alpha^s$ bounds a topological disc $C\subset\Wi S$. Consider two copies of $C$, and glue them along $\alpha^s$ in order to obtain a new topological disc $D$. The boundary of $D$ is the union of two copies of $\alpha^u$, hence is piecewise smooth. The foliation $\Wi F^s$ provides a one-dimensional foliation on $D$, which is topologically tranverse to boundary $\partial D$. This contradicts Poincar\'e-Hopf Theorem.
\end{proof}

The next three propositions below concern the action of the Poincar\'e map $\tilde\theta$ on the foliations $\Wi F^s,\Wi F^u$ and the laminations $\Wi L^s,\Wi L^u$. We recall that $\Wi L^s$ and $\Wi L^u$ are sub-laminations (\emph{i.e.} union of leaves) of the foliations~$\Wi F^s$ and $\Wi F^u$ respectively.

\begin{proposition}
\label{p.Poincare-map}
The Poincar\'e map $\tilde\theta:\Wi S-\Wi L^s\to\Wi S-\Wi L^u$ preserves the foliations $\Wi F^s$ and $\Wi F^u$.
\end{proposition}

\begin{remark}
\label{r.preserve-foliations}
Propostion~\ref{p.Poincare-map} states that the foliation $(\Wi F^s)_{|\Wi S-\Wi L^s}$ is mapped by $\tilde\theta$ to the foliation $(\Wi F^s)_{|\Wi S-\Wi L^u}$. The leaves of $(\Wi F^s)_{|\Wi S-\Wi L^s}$ are full leaves of the foliation $\Wi F^s$. On the contrary, a leaf of the foliation $(\Wi F^s)_{|\Wi S-\Wi L^u}$ is never a full leaf of $\Wi F^s$ (because every leaf of $\Wi F^s$ is ``cut into infinitely many pieces" by the transverse lamination $\Wi L^u$).  As a consequence, $\tilde\theta$ maps leaves of $\Wi F^s$ to pieces of leaves of $\Wi F^s$. 
Similarly,  $\tilde\theta$ maps pieces of leaves of $\Wi F^u$ to full leaves of $\Wi F^u$.
\end{remark}

\begin{proof}[Proof of Proposition~\ref{p.Poincare-map}]
Recall that $\Wi F^s$ is defined as the intersection of the foliation $\Wi\cF^s$ with the transverse surface $\Wi S$. The foliation $\Wi\cF^s$ is leafwise invariant under the flow $(\Wi Y^t)$. As a consequence, $\Wi F^s=\Wi\cF^s\cap\Wi S$ is invariant under the Poincar\'e return map of $(\Wi Y^t)$ on $\Wi S$.
\end{proof}

\begin{proposition}
\label{p.preimage-stable-lamination}
For every $n\geq 0$, $\bigcup_{p=0}^n\tilde \theta^{-p}(\Wi L^s)$ is a closed sub-lamination of the foliation $\Wi F^s$.
\end{proposition}

\begin{proof}
The foliation $\Wi F^s$ is invariant under the Poincar\'e map $\tilde\theta:\Wi S-\Wi L^s\to\Wi S-\Wi L^u$. Since $\Wi L^s$ is a union of leaves of $\Wi F^s$, it follows that $\tilde \theta^{-1}(\Wi L^s)$ is a union of leaves of $\Wi F^s$. Moreover, since $\Wi L^s$ is a closed subset of $\widetilde S$, its pre-image $\tilde\theta^{-1}(\Wi L^s)$ must be a closed subset of $\widetilde S-\Wi L^s$ (remember that $\tilde\theta$ is well-defined on $\widetilde S-\Wi L^s$). Therefore $\bigcup_{p=0}^1 \tilde\theta^{-p}(\Wi L^s)$ is a closed subset of $\Wi S$. So $\bigcup_{p=0}^1 \tilde\theta^{-p}(\Wi L^s)$ is a closed union of leaves of $\Wi F^s$, \emph{i.e.} a closed sub-lamination of $\Wi F^s$. Repeating the same arguments, one proves by induction that $\bigcup_{p=0}^n \tilde\theta^{-p}(\Wi L^s)$ is a closed sub-lamination of $\Wi F^s$ for every $n\geq 0$.
\end{proof}

\begin{proposition}
\label{p.stable-set-Lambda}
$\displaystyle \bigcup_{p=0}^\infty \tilde\theta^{-p}(\Wi L^s)=\Wi W^s(\Lambda)\cap\Wi S$
\end{proposition}

\begin{proof}
By definition, $W^s(\Lambda)\cap S$  is the set of all points $x\in S$ so that the forward orbit of $x$ converges towards the set $\Lambda$, which is disjoint from $S$. As a consequence, for every point $x\in W^s(\Lambda)\cap S$, the forward orbit of $x$ intersects the surface $S$ only finitely many times, say $p(x)$ times. We have observed that $L^s$ is the set of all points $y\in S$ so that the forward orbit of $y$ does not intersect $S$ and converges towards the set $\Lambda$ (see Subsection~\ref{ss.setting}). It follows that, for every  $x\in W^s(\Lambda)\cap S$, the last intersection point $\theta^{p(x)}$ of the forward orbit of $x$ with $S$ is in $L^s$. This proves the inclusion $W^s(\Lambda)\cap S\subset\bigcup_{p=0}^\infty \theta^{-p}( L^s)$. The converse inclusion is straightforward. Hence $\bigcup_{p=0}^\infty \theta^{-p}( L^s)=W^s(\Lambda)\cap S$. The equality $\bigcup_{p=0}^\infty \tilde\theta^{-p}(\Wi L^s)=\Wi W^s(\Lambda)\cap\Wi S$ follows by lifting to the universal cover.
\end{proof}

Of course, $\Wi W^s(\Lambda)\cap\Wi S$ and $\Wi W^u(\Lambda)\cap\Wi S$ are union of leaves of the foliations $\Wi F^s$ and $\Wi F^u$ respectively. But these sets are not closed. More precisely:

\begin{proposition}
\label{p.density-stable-manifold}
Both $\Wi W^s (\Lambda)\cap\Wi S$ and $\Wi S - \Wi W^s (\Lambda)$ are dense in $\Wi S$.
\end{proposition}

\begin{proof}
Recall that $(Y^t)$ is a transitive Anosov flow on $M$. Hence every leaf of the weak stable foliation $\cF^s$  is dense in $M$. Since both $W^s(\Lambda)$ and $M\setminus W^s(\Lambda)$ are non-empty union of leaves of the foliation $\cF^s$, and since the leaves of $\cF^s$ are transversal to the surface $S$, it follows that both $W^s(\Lambda)\cap S$ and $S\setminus W^s(\Lambda)$ are dense in $S$. Lifting to the universal cover, we obtain that $\Wi{W^s} (\Lambda)\cap\Wi S$ and $\Wi S - \Wi{W^s} (\Lambda)$ are dense in $\Wi S$.
 \end{proof}

Of course, the analogs of Propositions~\ref{p.preimage-stable-lamination},~\ref{p.stable-set-Lambda} and~\ref{p.density-stable-manifold} for $\widetilde L^u$ and $W^u(\Wi \Lambda)$ hold ($\tilde\theta^{-p}$ should be replaced by $\tilde\theta^p$ in Propositions~\ref{p.preimage-stable-lamination} and~\ref{p.stable-set-Lambda}). We will now describe the topology of the connected components of $\Wi S\setminus\Wi L^s$. We first introduce some vocabulary.

\begin{definition}
We call \emph{proper stable strip} every topological open disc $D$ of $\Wi S$ whose boundary is the union of two leaves of the foliation $\Wi F^s$.
\end{definition}

If $D$ is a proper stable strip, one can easily construct a homeomorphism $h$ from the closure of $D$ to $\RR\times [-1,1]$. We will need the following stronger notion.

\begin{definition}
\label{d.trivially-bifoliated}
We say that a proper stable strip $D$ is \emph{trivially bifoliated} if there exists a homeomorphism $h$ from the closure of $D$ to $\RR\times [-1,1]$ mapping the foliations $\Wi F^s$ and $\Wi F^u$ to the horizontal and vertical foliations on $\RR\times [-1,1]$.
\end{definition}

\begin{figure}[ht]
\begin{center}
  \includegraphics[totalheight=7cm]{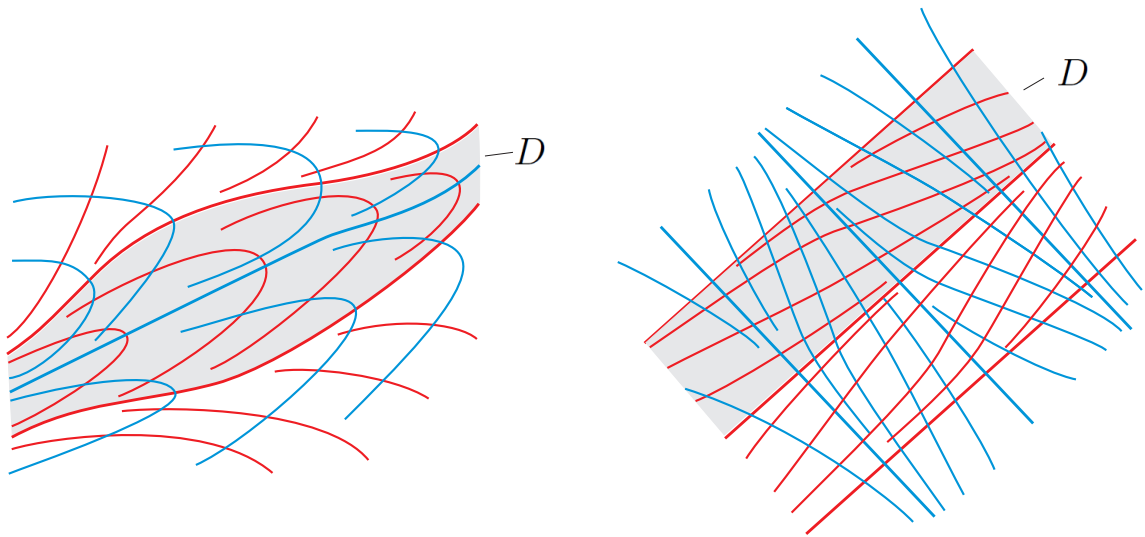}
\caption{\label{f.strip}A proper stable strip (left). A trivially bifoliated proper stable strip (right).}
 \end{center}
\end{figure}

Of course, \emph{proper unstable strips} and \emph{trivially bifoliated proper unstable strips} can be defined similarly. The proposition below gives a fairly precise description of the positions of the connected components of  $\Wi S-\Wi L^s$ with respect to the foliations $\Wi F^s$ and $\Wi F^u$.

\begin{proposition}
\label{p.strip}
Every connected component of $\Wi S-\Wi L^s$ is a trivially bifoliated proper stable strip bounded by two leaves of the lamination $\Wi L^s$.
\end{proposition}

\begin{proof}
Let $D$ be a connected component of $\Wi S-\Wi L^s$. Denote by $P$ the connected component of $\Wi S$ containing $D$. Since $P$ is a topological plane (Proposition~\ref{p.topology-surface}), and since each leaf of $\Wi L^s$ is a properly embedded topological line (Proposition~\ref{p.intersection-two-leaves}) which separates $P$ into two connected components, it follows that $D$ is a topological disc. The boundary of $D$ is a union of leaves of $\Wi L^s$ (which we call \emph{the boundary leaves of $D$}). We denote by $\overline D$ the closure of $D$.

\medskip

\noindent\textit{Claim 1. Let $\ell^u$ be a leaf of the foliation $\Wi F^u$ intersecting $\overline D$, and $\alpha^u$ be a connected component of $\ell^u\cap\overline D$. Then $\alpha^u$ is an arc joining two different boundary leaves of $D$.}

\smallskip

\noindent Let $R$ be a connected component of $D\setminus\Wi L^u$, so that $\alpha^u$ is included in the closure $\overline R$ of $R$ (actually $R$ is unique, but we will not use this fact). Observe that $R$ is a connected component of $\Wi S-(\Wi L^s\cup\Wi L^u)$. Our assumptions (namely, the strong transversality of the gluing map $\psi$) imply that $R$ is a relatively compact topological disc, whose boundary $\partial R$ is made of four arcs $\alpha_-^s,\alpha_+^s,\alpha_-^u,\alpha_+^u$, where $\alpha_-^s$ and $\alpha_-^s$ are disjoint and lie in some leaves of $\Wi L^s$, and where $\alpha_-^u$ and $\alpha_+^u$ are disjoint and lie in some leaves of $\Wi L^u$. Loosely speaking, $R$ is a rectangle with two sides $\alpha_-^s,\alpha_+^s$ in $\Wi L^s$ and two sides $\alpha_-^u,\alpha_+^u$ in $\Wi L^u$. Proposition~\ref{p.intersection-two-leaves} implies that
$\ell^u$ intersects $\alpha_-^s$ and $\alpha_+^s$ at no more than one point. Since $\ell^u$ is a proper line, and $\overline R$ is a compact set, it follows that $\alpha^u$ must be an arc going from $\alpha_-^s$ to $\alpha_+^s$. Using again Proposition~\ref{p.intersection-two-leaves}, it also follows that $\alpha_-^s$ to $\alpha_+^s$ cannot be in the same leaf of $\Wi F^s$. The claim is proved.

\medskip

\noindent\textit{Claim 2. $D$ has exactly two boundary leaves.}

\smallskip

\noindent In order to prove this claim, we endow the foliation $\Wi F^u$ with an orientation (this is possible since $\Wi F^u$ is a foliation on a collection of topological planes). For every $x\in \overline D$, we denote by $\ell^u(x)$ the leaf of the foliation $\Wi F^u$ passing through $x$, and denote by $\alpha^u(x)$ the connected component of $\ell^u _x\cap\overline D$ containing $x$. Note that $\ell^u(x)$ and $\alpha^u(x)$ are oriented by the orientation of $\Wi F^u$. By Claim~1, $\alpha^u(x)$ is an arc whose endpoints lie on two boundary leaves $\ell_-^s(x)$ and $\ell_+^s(x)$ of $D$. By transversality of the foliations $\Wi F^u$ and $\Wi F^s$, the maps $x\mapsto\ell^s_-(x)$ and $x\mapsto\ell^s_+(x)$ are locally constant. Since $\overline D$ is connected, these maps are constant. In other words, one can find two boundary leaves $\ell^s_-$ and $\ell^s_+$ of $D$, so that $\alpha^u(x)$ is an arc from $\ell^s_-$ to $\ell^s_+$ for every $x\in\overline D$. It follows that $\ell^s_-$ and $\ell^s_+$ are the only accessible boundary leaves of $D$: otherwise, one can consider another boundary leaf $\ell^s$, take a point $x\in\ell^s$, and get a contradiction since one end of $\alpha^u_x$ is on $\ell^s$. As a further consequence, the accessible boundary of $D$ is closed (recall that $\ell^s_-$ and $\ell^s_+$ are properly embedded lines), and therefore coincides with the boundary of $D$. We finally conclude that $\ell^s_-$ and $\ell^s_+$ are the only boundary leaves of $D$ and Claim 2 is proved.

\medskip

Claim~1 and~2 already imply that $D$ is a proper stable strip bounded by two leaves $\ell^s_-,\ell^s_+$ of $\Wi L^s$. We are left to prove that $D$ a trivially bifoliated. Recall that $\Wi S$ is a topological plane (Proposition~\ref{p.topology-surface}), and that $\ell_-^s,\ell_+^s$ are properly embedded topological lines (Proposition~\ref{p.intersection-two-leaves}). By easy planar topology, it folllows that there exists a homeomorphism $h$ from $\overline D$ to $\RR\times [-1,1]$ mapping $\ell_-^s$ and $\ell_+^s$ to $\RR\times \{-1\}$ and $\RR\times \{1\}$ respectively. Claim~1 implies that $h_*(\Wi F^u_{\overline D}$ is a foliation of $\RR\times [-1,1]$ by arcs going from $\RR\times \{-1\}$ and $\RR\times \{1\}$. One can easily construct a self-homeomoprhism $h'$ of $\RR\times [-1,1]$ mapping this foliation on the vertical foliation of $\RR\times [-1,1]$. Up to replacing $h$ by $h'\circ h$, we will assume that $h$ maps $\Wi F^u_{\overline D}$ on the vertical foliation of $\RR\times [-1,1]$. Now we consider a leaf $\ell^s$ of the foliation $\Wi F^s$ included in $\overline D$. According to Proposition~\ref{p.intersection-two-leaves}, $\ell^s$ intersects each leaf of $\Wi F^u$ at no more than one point. Hence $h(\ell^s)$ intersects each vertical segment in $\RR\times [-1,1]$ at no more than one point. Let $E$ be the set of $t\in\RR$ so that $h(\ell^s)$ intersects the vertical segment $\{t\}\times [-1,1]$. Since $\ell^s$ is a proper topological line tranvsersal to $\Wi F^u$, the set $E_t$ must be open and closed in $\RR$. Therefore $h(\ell^s)$ intersects every vertical segment in $\RR\times [-1,1]$ at exactly one point. In other words, the leaves of $h_*(\Wi F^s_{\overline D})$ are graphs over the first coordinate in $\RR\times [-1,1]$. One can easily modify the homeomorphism $h$ so that $h_*(\Wi F^s_{\overline D})$ is the horizontal foliation of $\RR\times [-1,1]$. Hence $D$ is a trivially bifoliated proper stable  strip.
\end{proof}

Of course, the unstable analog of Proposition~\ref{p.strip} holds true: every connected component of $\Wi S-\Wi L^u$ is a trivially bifoliated proper unstable strip bounded by two leaves of the lamination $\Wi L^u$. On the other hand, $\tilde\theta$ maps connected components of $\Wi S-\Wi L^s$ to connected component of $\Wi S-\Wi L^u$. So, we obtain:

\begin{corollary}
\label{c.image-connected-component}
If $D$ is a connected component of $\Wi S-\Wi L^s$, then $\tilde\theta(D)$ is a trivially bifoliated proper unstable strip, disjoint from $\Wi L^u$, bounded by two leaves of the lamination $\Wi L^u$.
\end{corollary}

The following proposition describes the action of $\tilde\theta$ on the connected components of $\Wi S-\Wi L^s$.

\begin{proposition}
\label{p.substrip}
Let $D$ be a connected component of $\Wi S-\Wi L$, and $D'$ be any trivially bifoliated proper stable strip. Assume that $D\cap\tilde\theta^{-1}(D')$ is non-empty. Then $D\cap\tilde\theta^{-1}(D')$ is a trivially bifoliated proper stable sub-strip of $D$.
\end{proposition}

\begin{proof}
We call \emph{trivially bifoliated rectangle} every topological open disc $R\subset\Wi S$ such that there exists a homeomorphism from the closure of $R$ to $[-1,1]^2$ mapping the restrictions of $\Wi F^s$ and $\Wi F^u$ to the horizontal and vertical foliations of $[-1,1]^2$. In particular, the boundary of such a trivially bifoliated rectangle is made of two stable sides, and two unstable sides.

According to Corollary~\ref{c.image-connected-component}, $\tilde\theta(D)$ is a trivially bifoliated proper unstable strip, disjoint from $\Wi L^u$, bounded by two leaves of $\Wi L^u$. By assumption, $D'$ is a trivially bifoliated proper stable strip. It easily follows that $\tilde\theta(D)\cap D'$ is a trivially bifoliated rectangle, disjoint from $\Wi L^u$, whose unstable sides are in $\Wi L^u$ (see Figure~\ref{f.substrip}). Observe that the interiors of two stable sides of $\tilde\theta(D)\cap D'$ are full leaves of $\Wi F^s_{|\Wi S-\Wi L^u}$. Hence:
\begin{itemize}
\item[$(\star)$] $\theta(D)\cap D'$ is a connected subset of $\theta(D)$, and the boundary of $\theta(D)\cap D'$ in $\theta(D)$ is made of two disjoint leaves of $\Wi F^s_{|\Wi S-\Wi L^u}$.
\end{itemize}
Now recall that $\tilde\theta^{-1}$ is a homeomorphism from $\Wi S-\Wi L^u$ to $\Wi S-\Wi L^s$, mapping leaves of $\Wi F^s_{|\Wi S-\Wi L^u}$ to full leaves of $\Wi F^s$  (see Proposition~\ref{p.Poincare-map} and Remark~\ref{r.preserve-foliations}). Also observe that $D\cap \tilde\theta^{-1}(D')$ is a subset of $D$. As a consequence, Properties~$(\star)$ implies:
\begin{itemize}
\item[$(\star')$] $D\cap \tilde\theta^{-1}(D')$ is a connected subset of  $D$, and the boundary of $D\cap \tilde\theta^{-1}(D')$ is made of two disjoint leaves of $\Wi F^s$.
\end{itemize}
Since $D$ is a trivially foliated proper stable strip $D$, Property~$(\star')$ clearly implies that $D\cap \tilde\theta^{-1}(D')$ is a trivially bifoliated proper stable sub-strip of $D$.  See Figure~\ref{f.substrip}.
\end{proof}

\begin{figure}[ht]
\begin{center}
  \includegraphics[totalheight=7cm]{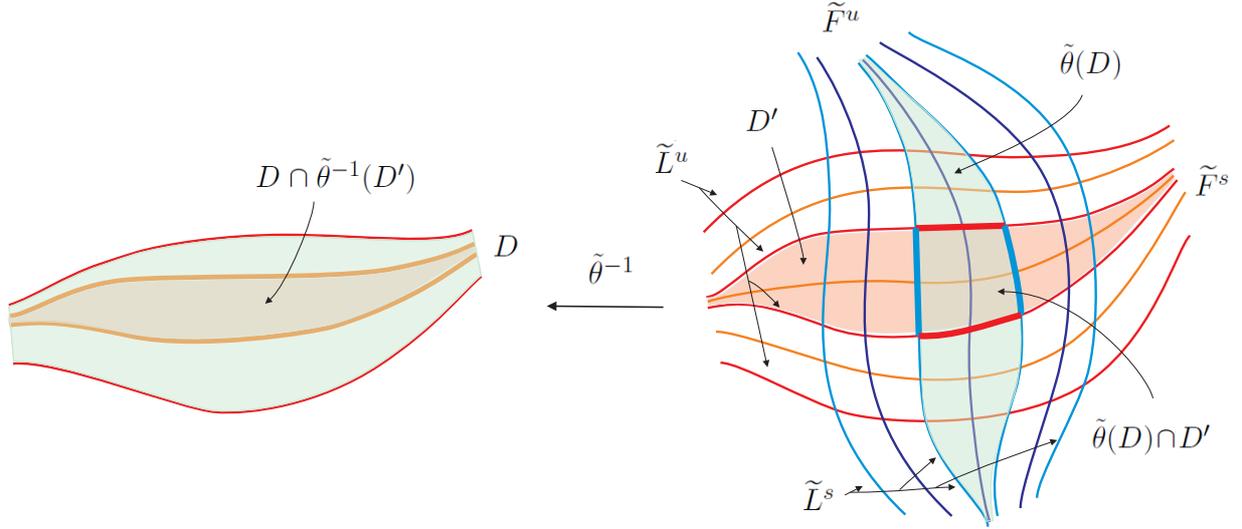}
\caption{\label{f.substrip}The proof of Proposition~\ref{p.substrip}}
\end{center}
\end{figure}

\subsection{The coding procedure}
\label{ss.coding-procedure}

In this section, we will use the connected components of  $\Wi S \setminus \Wi L^s$ to describe the itinerary of the orbits the flow $(\Wi Y^t)$ that do not belong to $\Wi W^s(\Lambda)\cup\Wi W^u(\Lambda)$. We consider the alphabet
$$\cA:=\{\mbox{connected components of }\Wi S \setminus \Wi L^s\},$$
and the symbolic spaces
 \begin{eqnarray*}
 \Sigma^s & = & \left\{\overline{D}^s=(D_p)_{p\geq 0}\mbox{ such that } \  D_p \in \cA \mbox{ and }\Wi{\theta} (D_p) \cap D_{p+1} \neq \emptyset\mbox{ for every }p\right\},\\
    \Sigma^u & = & \left\{\overline{D}^u=(D_p)_{p<0}\mbox{ such that } \  D_p \in \cA \mbox{ and }\Wi{\theta} (D_p) \cap D_{p+1} \neq \emptyset\mbox{ for every }p\right\},\\
     \Sigma & = & \left\{\overline{D}=(D_p)_{p\in \mathbb{Z}}\mbox{ such that } \  D_p \in \cA \mbox{ and }\Wi{\theta} (D_p) \cap D_{p+1} \neq \emptyset\mbox{ for every }p\right\}.
 \end{eqnarray*}

In order to define the coding maps, we need to introduce some leaf spaces. We will denote by $f^s$ the leaf space of the foliation $\Wi F^s$ (equipped with the quotient topology). We will denote $f^{s,\infty}$ the subset of $f^s$ made of the leaves that are not in $\Wi W^s(\Lambda)$. Similarly, we denote by $f^u$ the leaf space of $\Wi F^u$, and by $f^{u,\infty}$ the subset fo $f^u$ made of the leaves that are not in $\Wi W^u(\Lambda)$. Finally we denote by $\Wi S^\infty$ the set of points $\Wi S$ that are neither in $\Wi W^s(\Lambda)$ nor in $\Wi W^u(\Lambda)$.
\begin{eqnarray*}
f^{s,\infty} & = & \{\mbox{leaves of }\Wi F^s\mbox{ that are not in }\Wi W^s(\Lambda)\},\\
f^{u,\infty} & = & \{\mbox{leaves of }\Wi F^u\mbox{ that are not in }\Wi W^u(\Lambda)\},\\
\Wi S^\infty & = & \Wi S-(\Wi W^s(\Lambda)\cup\Wi W^u(\Lambda)).
\end{eqnarray*}

By Proposition~\ref{p.stable-set-Lambda}, if $\ell^s\in f^{s,\infty}$, then $\tilde\theta^p (\ell^s)$ is included in a connected component of $\Wi S-\Wi L^s$ for every $p\geq 0$. Similarly, if $\ell^u\in f^{u,\infty}$, then $\tilde\theta^p (\ell^u)$ is included in a connected component of $\Wi S-\Wi L^u$ for every $p\leq 0$. Since $\tilde\theta^{-1}$ maps homeomorphically $\Wi S-\Wi L^u$ to $\Wi S-\Wi L^s$, we deduce that: if $\ell^u\in f^{u,\infty}$, then $\tilde\theta^p (\ell^u)$ is included in a connected component of $\Wi S-\Wi L^s$ for every $p<0$. As a further consequence, if $x$ is a point of $\Wi S^\infty$, then $\tilde\theta^p (x)$ is in a connected component of $\Wi S-\Wi L^s$ for every $p\in\ZZ$. This shows that the following coding maps are well-defined:
$$\begin{array}{ccrcll}
\chi^s & : & f^{s,\infty} & \longrightarrow & \Sigma^s & \\
& & \ell^s & \longmapsto &  \overline{D}^s=(D_p)_{p\geq 0} & \mbox{where }\tilde\theta^p (\ell^s) \subset D_p\mbox{ for every }p\geq 0\\
&\\
\chi^u & : & f^{u,\infty} & \longrightarrow & \Sigma^u & \\
& & \ell^u & \longmapsto &  \overline{D}^u=(D_p)_{p<0} & \mbox{where }\tilde\theta^p (\ell^u) \subset D_p\mbox{ for every }p<0\\
&\\
\chi & : & \Wi S^\infty & \longrightarrow & \Sigma & \\
& & x & \longmapsto &  \overline{D}=(D_p)_{p\in\ZZ} & \mbox{where }\tilde\theta^p (x) \in D_p\mbox{ for every }p\in\ZZ
\end{array}$$

The following proposition is an important step ingredient of the proof of Theorem~\ref{t.main}.

\begin{proposition}
\label{p.bijective}
The maps $\chi^s$, $\chi^u$ and $\chi$ are bijective.
 \end{proposition}

 \begin{lemma}
 \label{l.inverse-code}~
 \begin{enumerate}
 \item For every $\overline{D}^s=(D_p)_{p\geq 0}\in\Sigma^s$, the set $\bigcap_{p\geq 0}\tilde \theta^{-p} (D_p)$ is a stable leaf $\ell^s\in f^{s,\infty}$.
 \item For every $\overline{D}^u=(D_p)_{p<0}\in\Sigma^u$, the set $\bigcap_{p<0}\tilde \theta^{-p} (D_p)$ is an unstable leaf $\ell^u\in f^{u,\infty}$.
 \item For every $\overline{D}=(D_p)_{p\in\ZZ}\in\Sigma$, the set $\bigcap_{p\in\ZZ}\tilde \theta^{-p} (D_p)$ is a single point $x\in \Wi S^\infty$.
 \end{enumerate}
 \end{lemma}

\begin{remark}
Lemma~\ref{l.inverse-code} is completely false if we replace the connected components of $\Wi S\setminus\Wi L^s$ by the connected components of $S\setminus L^s$ (and $\tilde\theta$ by $\theta$). For example, if $(D_p)_{p\geq 0}$ is a sequence of connected components of $S\setminus L^s$, then $\bigcap_{p\geq 0}\theta^{-p} (D_p)$, if not empty, will be the union of uncountably many leaves of the foliation $F^s$. This is the reason why, we need to work in the universal cover of $M$.
\end{remark}

 \begin{proof}
Let us prove the first item. Consider a sequence $\overline{D}^s=(D_p)_{p\geq 0}\in \Sigma^s$. By Proposition~\ref{p.strip}, $D_0$ is a trivially bifoliated proper stable strip. Proposition~\ref{p.substrip} and a straightfoward induction imply that for every $n\in \mathbb{N}$, the set $\bigcap_{p=0}^n\Wi\theta^{-p} (D_p)$ is a sub-strip of $D_0$. So $(\bigcap_{p=0}^n\Wi\theta^{-p} (D_p))_{n\geq 0}$ is a decreasing sequence of sub-strips of the trivially bifoliated proper stable strip $D_0$. It easily follows that $\bigcap_{p\geq 0}\Wi\theta^{-p} (D_p)$ is a sub-strip of $D_0$. In particular, $\bigcap_{p\geq 0}\Wi\theta^{-p} (D_p)$ is a connected union of leaves of $\Wi F^s$. On the other hand, since $D_0,D_1,\dots$ are connected components of $\Wi S-\Wi L^s$, the set $\bigcap_{\geq 0}\Wi\theta^{-p} (D_p)$ is disjoint from $\bigcup_{p\geq 0}\Wi\theta^{-p} (\Wi L^s)=\Wi W^s(\Lambda)\cap \Wi S$ (see Proposition~\ref{p.stable-set-Lambda}). But $\Wi W^s(\Lambda)\cap \Wi S$ is dense in $\Wi S$ (Proposition~\ref{p.density-stable-manifold}). It follows that $\bigcap_{p\geq 0}\Wi{\theta}^{-p} (D_p)$ must be a single leaf of $\Wi F^s$. This completes the proof of item~1.

Item~2 follows from exactly the same arguments as item~1. In order to prove the last item, we consider a sequence $\overline{D}=(D_p)_{p\in\ZZ}$ in $\Sigma$. According to the items~1 and~2, $\bigcap_{p\geq 0}\Wi \theta^{-p} (D_p)$ is a leaf $\ell^s$ of the foliation $\Wi F^s$ and $\bigcap_{p< 0}\Wi \theta^{-p} (D_p)$ is a leaf $\ell^u$ of the foliation $\Wi F^u$. Since $\overline{D}=(D_p)_{p\in\ZZ}$ is in $\Sigma$, the intersection $D_0\cap\tilde\theta^(D_{-1})$ is not empty. Since $D_0$ is a trivially bifoliated proper stable strip (Proposition~\ref{p.strip}) and $\tilde\theta^(D_{-1})$ is a trivially bifoliated proper unstable strip (Corollary~\ref{c.image-connected-component}), every leaf of $\Wi F^s$ in $D_0$ must intersects every leaf of $\Wi F^u$ in  $\tilde\theta(D_{-1})$ at exactly one point. In particular, $\bigcap_{p\in\ZZ}\Wi \theta^{-p} (D_p)=\ell^s\cap\ell^u$ is made of exactly one point $x$. Since the leaves $\ell^s$ and $\ell^u$ are disjoint from $\Wi W^s(\Lambda)$ and $\Wi W^u(\Lambda)$ respectively, the point $x$ must be in $\Wi S^\infty$. This completes the proof.
\end{proof}

\begin{proof}[Proof of Proposition~\ref{p.bijective}]
Lemma~\ref{l.inverse-code} allows to define some inverse maps for $\chi^s$, $\chi^u$ and $\chi$. Therefore, $\chi^s$, $\chi^u$ and $\chi$ are bijective.
\end{proof}

Deck transformation preserve the surface $\widetilde S$, the foliations $\Wi \cF^s, \Wi \cF^u$, and the laminations $W^s(\Wi \La), W^u(\Wi \La)$. This induces some natural actions of $\pi_1(M)$ on the set $\Wi S^\infty$, on the leaf spaces $f^{s,\infty}, f^{u,\infty}$, on the alphabet $\cA$, and therefore on the symbolic spaces $\Sigma,\Sigma^s,\Sigma^u$. From the definition of the coding maps, one easily checks that:

\begin{proposition}
\label{p.coding-commutes-pi-1}
The coding maps $\chi$, $\chi^s$ and $\chi^u$ commute with the actions of the fundamental group of $M$ on $\Wi S^\infty$ $f^s$, $f^u$, $\Sigma$ $\Sigma^s$ and $\Sigma^u$.
\end{proposition}

The definition of the coding maps also implies that:

\begin{proposition}
\label{p.coding-conjugates-Poincare-shift}
The coding map $\chi$ (resp. $\chi^s$ and $\chi^u$) conjugates the action of Poincar\'e first return map $\Wi\theta$ on $\Wi S^\infty$ (resp. $f^s$ and $f^u$) to the left shift on the symbolic space $\Sigma$ (resp. $\Sigma^s$ and $\Sigma^u$).
\end{proposition}

Given an integer $n\geq 0$ and some connected components $D_0^0,\dots,D_n^0$ of $\Wi S-\Wi L^s$, we define the cylinder
$$[D_0^0\dots D_n^0]^s:=\{(D_p)_{p\geq 0}\in\Sigma^s\mbox{ such that } D_p=D_p^0\mbox{ for }0\leq p\leq n\}.$$
Similarly, given $n<0$ and some connected components  $D_n^0,\dots,D_{-1}^0$  of $\Wi S-\Wi L^s$, we define the cylinder
$$[D_n^0\dots D_{-1}^0]^u:=\{(D_p)_{p< 0}\in\Sigma^u\mbox{ such that } D_p=D_p^0\mbox{ for }n\leq p\leq -1\}.$$
The following proposition will be used in the next subsection.

\begin{proposition}\label{p.sub-strips}~
\begin{enumerate}
  \item for $n\geq 0$ and $D_0,\dots,D_n\in\cA$, the set $(\chi^s)^{-1}([D_0 D_1 \dots D_n]^s)= \bigcap_{0\leq p\leq n}\tilde{\theta}^{-p} (D_p)$ is either empty or a sub-strip of the trivially foliated proper stable strip $D_0$ bounded by two leaves of $\Wi \theta^{-n}(\Wi L^s)$;
  \item for $n< 0$ and $D_n,\dots,D_{-1}\in\cA$, the set $(\chi^u)^{-1}([D_{n} D_{n+1} \dots D_{-1}])= \bigcap_{-n\leq p\leq -1}\tilde{\theta}^{-p+1} (D_p)$ is a sub-strip of the trivially foliated proper unstable strip $\tilde\theta(D_{-1})$ bounded by two leaves of $\Wi \theta^{K-1}(\Wi{L^u})$.
 \end{enumerate}
\end{proposition}

\begin{proof}
This follows from the arguments of the proof of Lemma~\ref{l.inverse-code}.
\end{proof}

\subsection{Partial orders on the leaf spaces and the symbolic spaces}
\label{ss.partial-orders}

We will now describe a partial pre-order on the leaf space $f^s$. The preservation of this partial pre-order will be a fundamental ingredient of our proof of Theorem~\ref{t.main} in Section~\ref{s.top-equiv}.

Let us start by choosing some orientations. First of all, we choose an orientation of the hyperbolic plug $U$. The orientation of $U$, together with the vector field $X$, provide an orientation of $\partial U$: if $\omega$ is a $3$-form defining the orientation on $U$, then the $2$-form $i_X U$ defines the orientation on $\partial U$. The orientation of $U$ induces an orientation of the manifold $M=U/\psi$ (we have assumed that the manifold $M$ is orientable, which is equivalent to assuming that the gluing map $\psi$ preserves the orientation of $\partial U$), and the orientation of $\partial U$ induces an orientation of the surface $S=\pi(\partial^{in} U)=\pi(\partial^{out} U)$. The orientation of $M$ and $S$ induce some orientations on $\widetilde M$ and $\widetilde S$. Now, since every connected component of $\Wi S$ is a topological plane, the foliation $\Wi F^s$ is orientable. We fix an orientation of $\Wi F^s$. This automatically induces an orientation of the foliation $\Wi F^u$ as follows: the orientation of $\Wi F^u$ is chosen so that, if $Z^s$ and $Z^u$ are vector fields tangent to $\Wi F^s$ and  $\Wi F^u$ respectively, and pointing in the direction of the orientation of the leaves, then the frame field $(Z^s,Z^u)$ is positively oriented with respect to the orientation of $\Wi S$.

\begin{remarks}
\label{r.orientations}
\begin{enumerate}
\item By construction, the orientations of the manifold $\Wi M$ and the surface $\Wi S$ are related as follows: if $\omega$ is a $3$-form defining the orientation on $\Wi M$, then the $2$-form $i_{\Wi Y} \Wi M$ defines the orientation on $\Wi S$. As a consequence, the Poincar\'e return map $\tilde\theta$ of the orbits of $\Wi Y$ on $\Wi S$ preserves the orientation of $\Wi S$.
\item Consequently, for any connected component $D$ of $\Wi S-\Wi L^s$, if the Poincar\'e map $\tilde\theta_{|D}$ preserves (resp. reverses) the orientation of the foliation $\Wi F^s$, then it also preserves (resp. reverses) the orientation of the foliation $\Wi F^u$.
\end{enumerate}
\end{remarks}

Let $\ell$ be a leaf of the foliation $\Wi F^s$, contained in a  connected component $\Wi S_\ell$  of $\Wi S$. Recall that $\Wi S_\ell$ is a topological plane, and $\ell$ is a properly embedded line in $\Wi S_\ell$. As a consequence, $\Wi S_\ell\setminus\ell$ has two connected components.

\begin{definition}
We denote by $L(\ell)$ and $R(\ell)$ the two connected component of $\Wi S\setminus\ell$ so that the oriented leaves of $\Wi F^u$ crossing $\ell$ go from $L(\ell)$ towards $R(\ell)$. The points of $L(\ell)$ are said to be \emph{on the left} of $\ell$; the points of $R(\ell)$ are said to be \emph{on the right} of $\ell$
\end{definition}

 Now we can define a pre-order on the leaf space $f^s$.

 \begin{definition}[Pre-order on $f^s$]
 \label{d.order}
 Given two leaves $\ell\neq\ell'$ of the foliation $\Wi F^s$, we write $\ell \prec \ell'$ if there exists an arc of a leaf of $\Wi F^u$ with endpoints $a\in\ell$ and $a'\in\ell'$, so that the orientation of $\Wi F^u$ goes from $a$ towards $a'$.
 \end{definition}

\begin{proposition}
\label{p.pre-order}
$\prec$ is a pre-order on $f^s$: the relations $\ell \prec\ell'$ and $\ell' \prec \ell$ are incompatible
\end{proposition}

\begin{proof}
The relation $\ell\prec\ell'$ implies that the leaf $\ell'$ is on the right of $\ell$, that is $\ell'\subset R(\ell)$. Similarly,  the relation $\ell\prec\ell'$ implies $\ell'\subset L(\ell)$. The proposition follows since $L(\ell)\cap R(\ell)=\emptyset$.
\end{proof}

The proposition below is very easy to prove, but fundamental:

\begin{proposition}
\label{p.local-total-order}
$\prec$ is a local total order on $f^s$: for every leaf $\ell_0$ of $\Wi F^s$, there exists a neighbourhood $\mathcal{V}_0$ of $\ell_0$ in $f^s$ so that any two different leaves $\ell,\ell'\in\mathcal{V}_0$ are comparable (\emph{i.e.} satisfy either $\ell\prec\ell'$ or $\ell'\prec\ell$).
\end{proposition}

\begin{proof}
Consider a leaf $\ell_0$ of $\Wi F^s$ and a leaf $\ell^u$ of $\Wi F^u$ so that $\ell^u\cap\ell_0\neq\emptyset$. By transversality of the foliations $\Wi F^s$ and $\Wi F^u$, there exists a neighbourhood $\mathcal{V}_0$ of $\ell_0$ in $f^s$ so that $\ell^u$ crosses every leaf in $\mathcal{V}_0$. As a consequence, any two different leaves $\ell,\ell'\in\mathcal{V}_0$ are comparable for the pre-order $\prec$.
\end{proof}

The proposition below shows that the pre-order $\prec$ is  ``compatible" with the connected components decomposition of $\Wi S-\Wi L^s$.

\begin{proposition}
\label{p.order-compatible-cc}
Given two different elements $D,D'$ of $\cA$, the following are equivalent:
\begin{enumerate}
\item there exists some leaves $\ell_0,\ell_0'\in f^s$ so that $\ell_0\subset D$, $\ell_0' \subset D'$ and $\ell_0\prec\ell_0'$~;
\item every leaves $\ell,\ell'\in f^s$ so that $\ell\subset D$ and $\ell' \subset D'$ satisfy $\ell \prec \ell'$.
\end{enumerate}
\end{proposition}

\begin{proof}
Assume that 1 is satisfied. Since $\ell_0\prec\ell_0'$, there must be a leaf $\ell^u$ of the foliation $\Wi F^u$ intersecting both $\ell_0$ and $\ell_0'$. Proposition~\ref{p.strip} implies that $\alpha:=\ell^u\cap D$ and $\alpha':=\ell^u\cap D'$ are two disjoint arcs in the leaf $\ell^u$. Consider some leaves $\ell,\ell'$ of $\Wi F^s$ contained in $D$ and $D'$ respectively. Again Proposition~\ref{p.strip} implies that $\ell$ intersects $\ell^u$ at some point $a_\ell\in\alpha$ and $\ell'$ intersects $\ell^u$ at some point $a_{\ell'}\in\alpha'$. Since $\ell_0\preceq\ell_0'$  the orientation of $\ell^u$ goes from $\alpha$ towards $\alpha'$, hence from $a_\ell$ towards $a_{\ell'}$. This shows that $\ell\prec\ell'$.
\end{proof}

\begin{definition}[Pre-order on $\mathcal{A}$]
Given two different elements $D,D'$ of $\cA$, we write $D\prec D'$ if there exists some leaves $\ell_0,\ell_0'\in f^s$ so that $\ell_0\subset D$, $\ell_0' \subset D'$ and $\ell_0\prec\ell_0'$.
\end{definition}

\begin{definition}[Pre-order on $\Sigma^s$]
\label{d.order-symbolic}
The partial pre-order $\prec$ on $\mathcal{A}$ induces a lexicographic partial pre-order on $\Sigma^s\subset \mathcal{A}^{\NN}$ that will also be denoted by $\prec$~: for $\overline{D}=(D_p)_{p\geq 0}$ and $\overline{D'}=(D_p')_{p\geq 0}$ in $\Sigma^s$, we write $\overline{D} \prec \overline{D'}$ if and only if there exists $p_0\geq 0$ such that $D_p = D_p'$ for $p\in\{0,\dots,p_0-1\}$, $D_{p_0} \prec D_{p_0}'$.
\end{definition}

We have defined a pre-order on the leaf space $f^s$ (Definition~\ref{d.order}) and a pre-order on the symbolic space $\Sigma^s $ (Definition~\ref{d.order-symbolic}). It is natural to wonder whether the coding map $\chi^s:f^{s,\infty}\to \Sigma^s$ is compatible with these pre-orders or not. For pedagogical reason, we first consider the simple situation where the two-dimensional foliation $\cF^u$ is orientable:

\begin{proposition}\label{p.orioricoh}
Assume that the unstable foliation $\cF^u$ is orientable. Then the coding map $\chi^s: f^{s,\infty} \to \Sigma^s$ preserves the pre-orders, \emph{i.e.} for $\ell,\ell'\in f^{s,\infty}$, $\ell \prec\ell'$ if and only if $\chi^s(\ell) \prec \chi^s(\ell')$.
\end{proposition}

\begin{proof}
Since the two-dimension foliation $\cF^u$ is orientable, its lift $\Wi\cF^u$ is also orientable. Recall that the vector field $\Wi Y$ is tangent to the leaves of the foliation $\Wi \cF^u$. So, the orientatibility of the two-dimensional foliation $\Wi \cF^u$ implies that the return map $\tilde\theta$ of the orbits of the vector field $\Wi Y$ on the surface $\Wi S$ preserves the orientation of the one-dimensional foliation $\Wi F^u=\Wi\cF^u\cap\Wi S$.

Consider two leaves $\ell,\ell'\in f^{s,\infty}$ so that $\ell \prec\ell'$. Let $\chi^s (\ell)=(D_p)_{p\geq 0}$ and  $\chi^s (\ell)=(D_p')_{p\geq 0}$. Recall that this means that 
$$\ell=\bigcap_{p\geq 0}\Wi\theta^{-p}(D_p)\quad\mbox{ and }\quad\ell'=\bigcap_{p\geq 0}\Wi\theta^{-p}(D_p').$$ Consider the integer $p_0=\min\{p\geq 0\,,\, D_p\neq D_p'\}$ and the set 
$$\widehat D:=\bigcap_{p=0}^{p_0-1} \Wi\theta^{-p}(D_p).$$ 
Both the leaves $\ell$ and $\ell'$ are included in $\widehat D$, and, according to Proposition~\ref{p.sub-strips}, $\widehat D$ is a trivially bifoliated proper stable strip. So we can consider an arc $\alpha^u$ of a leaf $\ell^u$ of the foliation $\Wi F^u$, so that $\alpha^u$ is included in the trivially bifoliated proper stable strip $\widehat D$, so that the ends $a,a'$ of $\alpha^u$ are on $\ell$ and $\ell'$ respectively. Since $\ell\prec \ell'$, the orientation of $\Wi F^u$ goes from $a$ towards $a'$. Now observe that $\widehat D$ is a connected component of $\Wi S-\bigcup_{p=0}^{p_0-1} \Wi\theta^p(\Wi L^s)$. As a consequence, the map $\Wi\theta^{p_0}$ is well-defined on $\Wi D$. In particular, we can consider $\beta^u:=\Wi\theta^{p_0}(\alpha^u)$. Observe that $\beta^u$ is an arc of a leaf of the foliation $\Wi F^u$. Its ends $b:=\Wi\theta^{p_0}(a)$ and $b':=\Wi\theta^{p_0}(a')$ are respectively in $\Wi\theta^{p_0}(\ell)\subset D_{p_0}$ and $\Wi\theta^{p_0}(\ell')\subset D_{p_0}'$. Since the return map $\Wi\theta^{p_0}$ preserves the orientation of the foliation $\Wi F^u$, the orientation of $\Wi F^u$ goes from $b$ towards~$b'$. It follows that $\Wi\theta^{p_0}(\ell)\prec \Wi\theta^{p_0}(\ell')$ and therefore $D_{p_0}\prec D_{p_0}'$. As a further consequence,
$$\chi^s(\ell)=(D_0,D_1,\dots,D_{p_0-1},D_{p_0},\dots)\prec (D_0,D_1,\dots,D_{p_0-1},D_{p_0}',\dots)=\chi^s(\ell').$$
This completes the proof of the implication $\ell\prec\ell' \Rightarrow \chi^s(\ell)\prec \chi^s(\ell')$. The converse implication follows from the very same arguments in reversed order.
\end{proof}

In general, the relationships between the order on the leaf space $f^s$ and the symbolic space $\Sigma^s$ is more complicated:

\begin{proposition}\label{p.oricoh}
Let $\ell,\ell'$ be two different elements of $f^{s,\infty}$. Let $(D_p)_{p\geq 0}:=\chi^s (\ell)$ and $(D_p')_{p\geq 0}:=\chi^s (\ell')$. Let $p_0$ be the smallest interger $p$ so that $D_p\neq D_p'$.
\begin{enumerate}
        \item If the map $\Wi{\theta}^{p_0}_{|\bigcap_{p=0}^{p_0-1} \Wi\theta^{-p}(D_p)}$ preserves the orientation of the foliation $\Wi{F}^u$, then $$(\ell\prec\ell')\Longleftrightarrow (D_{p_0}\prec D_{p_0}')\Longleftrightarrow (\chi^s (\ell)\prec\chi^s (\ell')).$$
        \item If the map $\Wi{\theta}^{p_0}_{|\bigcap_{p=0}^{p_0-1} \Wi\theta^{-p}(D_p)}$ reverses the orientation of the foliation $\Wi{F}^u$, then $$(\ell\prec\ell')\Longleftrightarrow  (D_{p_0}'\prec D_{p_0})\Longleftrightarrow(\chi^s (\ell')\prec\chi^s (\ell)).$$
 \end{enumerate}
\end{proposition}

\begin{proof}
The arguments are exactly the same as in the proof of Proposition~\ref{p.orioricoh}.
\end{proof}

\section{Topological equivalence of Anosov flows}
\label{s.top-equiv}

We will now prove Theorem~\ref{t.main} with the help of the coding procedure implemented in section~\ref{s.coding}. 

\subsection{A simplification}

We begin by explaining why it is enough to prove Theorem~\ref{t.main} in  the particular case where the vector fields $X_1$ and $X_2$ coincide.

Let $(U,X_1,\psi_1)$ and $(U,X_2,\psi_2)$ be two triple satisfying the hypotheses of Theorem~\ref{t.main}. In particular $(U,X_1,\psi_1)$ and $(U,X_2,\psi_2)$ are strongly isotopic. This means that there exists a continuous one-parameter family $\{(U,X_t,\psi_t)\}_{t\in [1,2]}$ so that $(U,X_t)$ is a hyperbolic plug and $\psi_t:\partial^{out} U\to\partial^{in} U$ is a strongly transverse gluing map for every $t$. By standard hyperbolic theory, hyperbolic plugs are structurally stable. Hence, this means that we can find a continuous family $(h_t)_{t\in [1,2]}$ of self-homeomorphisms of $U$ so that $h_1=\mathrm{Id}$ and so that $h_t$ induces an orbital equivalence between $X_1$ and $X_t$. For $t\in [1,2]$, define
$$\widehat\psi_t:=(h_{t|\partial^{in} U})^{-1}\circ\psi_t\circ (h_{t|\partial^{out} U})$$ 
and observe that $\widehat\psi_1=\psi_1$. For sake of clarity, let $X:=X_1$. Then, 
\begin{itemize}
\item[--] the triples $(U,X,\widehat\psi_1)$ and  $(U,X,\widehat\psi_2)$ are strongly isotopic: the strong isotopy is given by the continuous path $\{(U,X,\widehat\psi_t)\}_{t\in [1,2]}$;
\item[--] for $t\in [1,2]$, the flow induced by the vector field $X$ on the manifold $\widehat M_t:=U/\widehat\psi_t$ is orbitally equivalent to the flow induced by the vector field $X_t$ on the manifold $M_t:=U/\psi_t$: the orbital equivalence is induced by the homeomorphism $h_t$. 
\end{itemize}
This shows that the hypotheses and the conclusion of Theorem~\ref{t.main} are satisfied for the triple $(U,X_1,\psi_1)$ and  $(U,X_2,\psi_2)$ if and only if they are satisfied for the triples $(U,X,\widehat \psi_1)$ and $(U,X,\widehat \psi_2)$. This allows us to replace the vector fields $X_1$ and $X_2$ by a single vector field $X$ in the proof of Theorem~\ref{t.main}.

\subsection{Setting}

From now on, we consider a hyperbolic plug $(U,X)$ endowed with two strongly transverse gluing diffeomorphisms $\psi_1,\psi_2:\partial^{out} U\to\partial^{in} U$. We denote by $\Lambda:=\bigcap_{t\in\RR} X^t(U)$ the maximal invariant set of the plug $(U,X)$. For $i=1,2$, the quotient space $M_i:=U/\psi_i$ is a closed three-dimensional manifold, and $X$ induces a vector field $Y_i$ on $M_i$. We assume that the hypotheses of Theorem~\ref{t.main} are satisfied, that is 
\begin{enumerate} 
\item[0.] the manifolds $U$, $M_1$ and $M_2$ are orientable,
\item[1.] for $i=1,2$, the flow $(Y_i^t)$ of the vector field $Y_i$ is a transitive Anosov flow,
\item[2.] the gluing maps $\psi_1$ and $\psi_2$ are strongly isotopic, \emph{i.e.} that there exists an isotopy $(\psi_s)_{s\in [1,2]}$ such that, for every $s$, the laminations $L^s$ and $\psi_s(L^u_X)$ are strongly transverse.
\end{enumerate}
In order to prove Theorem~\ref{t.main}, we have to construct a homeomorphism $H:M_1\to M_2$ mapping the oriented orbits of  the Anosov flow $(Y_1^t)$ to the orbits of the Anosov flow $(Y_2^t)$. The construction will be divided into several steps. 

\subsection{Starting point of the construction: the diffeomorphisms $\phi_{in},\phi_{out}:S_1\to S_2$}

For $i=1,2$, we denote by $\pi_i$ the projection of $U$ on the closed three-dimensional manifold $M_i=U/\psi_i$. We denote by 
$$S_i=\pi_i(\partial^{in}U)=\pi_i(\partial^{out}U)$$ 
the projection of the boundary of $U$. The surface $S_i$ is endowed with the strongly transverse laminations 
$$L^s_i:=\pi_i(L^s_X)\quad\quad\mbox{ and }\quad\quad L^u_i:=\pi_i(L^u_X).$$ 
The maps $\pi_{i|\partial^{in} U}:\partial^{in} U\to S_i$ and $\pi_{i|\partial^{out} U}:\partial^{out} U\to S_i$ are invertible. This provides us with two diffeomorphisms 
$$\phi_{in}:=\pi_{2|\partial^{in} U}\circ (\pi_{2|\partial^{in} U})^{-1}:S_1\to S_2\quad\quad\mbox{ and }\quad\quad\phi_{out}:=\pi_{2|\partial^{out} U}\circ (\pi_{2|\partial^{out} U})^{-1}:S_1\to S_2.$$ 

The diffeomorphisms $\phi_{in}$ and $\phi_{out}$ are the starting point of our construction. Observe that, at this step, we are very far from getting an orbital equivalence. Indeed, $\phi_{in}$ and $\phi_{out}$ are in no way compatible with the actions of the flows $(Y_1^t)$ and $(Y_2^t)$ (\emph{i.e.} they do not conjugate the Poincar\'e return maps of $(Y_1^t)$ and $(Y_2^t)$ on the surface $S_1$ and $S_2$). 

Nevertheless, the definition of the diffeomorphisms $\phi_{in}$ and $\phi_{out}$ imply that they satisfy:
$$\phi_{in}(L^s_1)=L^s_2\quad\mbox{and}\quad\phi_{out}(L^u_1)=L^u_2.$$ 

\begin{remark}
Be careful: in general $\phi_{in}(L^u_1)\neq L^u_2$ and $\phi_{out}(L^s_1)\neq L^s_2$.
\end{remark} 

On the other hand, the strong isotopy connecting the gluing maps $\psi_1$ and $\psi_2$ can be used to construct an isotopy between the diffeomorphisms $\phi_{in}$ and $\phi_{out}$: 

\begin{proposition}
\label{p.isotopy}
There exists a continuous family $(\phi_t)_{t\in [0,1]}$ of diffeomorphisms from $S_1$ to $S_2$, such that $\phi_0=\phi_{out}$, such that $\phi_1=\phi_{in}$, and such that the laminations $\phi_{t}(L^u_1)$ and $L^s_2$ are strongly transverse for every $t$. 
\end{proposition}

\begin{proof}
By assumption, the gluing maps $\psi_1$ and $\psi_2$ are connected by a continuous path $(\psi_s)_{s\in [1,2]}$ of diffeomorphisms from $\partial^{out} U$ to $\partial^{in} U$, such that the laminations $\psi_s(L^u)$ and $L^s$ are strongly transverse for every $s$. For $t\in [0,1]$, we set 
$$\phi_t:=\pi_{2|\partial^{out} U}\circ \psi_2^{-1}\circ \psi_{2-t}\circ (\pi_{1|\partial^{out} U})^{-1}.$$
From this formula, we immediately get 
$$\phi_0=\pi_{2|\partial^{out} U}\circ (\pi_{1|\partial^{out} U})^{-1}=\phi_{out}.$$ 
Plugging the equality $\pi_{i|\partial^{in} U}\circ \psi_i=\pi_{i|\partial^{out} U}$ into the definition of $\phi_1$, we get 
$$\phi_1= \pi_{2|\partial^{out} U}\circ \psi_2^{-1}\circ \psi_{1}\circ (\pi_{1|\partial^{out} U})^{-1}=\pi_{2|\partial^{in} U}\circ (\pi_{1|\partial^{in} U})^{-1}=\phi_{in}.$$
We know that the laminations $L^s_X$ and $\psi_{2-t}(L^u_X)$ are strongly transverse for every $t$. As a consequence, the laminations 
$$\pi_{2|\partial^{out} U}\circ \psi_2^{-1}(L^s_X)=\pi_{2|\partial^{in} U} (L^s_X)=L^s_2$$
and 
$$\pi_{2|\partial^{out} U}\circ \psi_2^{-1}\circ \psi_{2-t}(L^u_X)=\phi_t\circ \pi_{1|\partial^{out}U}(L^u_X)=\phi_{t}(L^u_1)$$ 
are strongly transverse for every $t$. This completes the proof.
\end{proof}

It is important to observe that the diffeomorphism $\phi_{in}$ can be obtained as the restriction of a diffeomorphism from $M_1$ to $M_2$:

\begin{proposition} 
\label{p.extension}
The diffeomorphism $\phi_{in}:S_1\to S_2$ is the restriction of a diffeomorphism $\Phi_{in}:M_1\to M_2$.
\end{proposition}

\begin{proof}
Once again, we use the existence of a continous path $(\psi_s)_{s\in [1,2]}$ of diffeomorphisms from $\partial^{out} U$ to $\partial^{in} U$ connecting the gluing maps $\psi_1$ and $\psi_2$. We consider a collar neighbourhood $V$ of $\partial^{out} U$ in $U$, and a diffeomorphism $\xi:\partial^{out} U\times [0,1]\to V$ of $V$ such that $\xi(\partial^{out} U\times \{0\})=\partial^{out} U$. We define a diffeomorphism $\bar \Phi_{in}:U\to U$ by setting $\bar \Phi_{in}(\xi(x,t)):=\psi_{2-t}^{-1}\circ\psi_1(x)$ for every $(x,t)\in \partial^{out} U\times [0,1]$, and $\bar \Phi_{in}=\mathrm{Id}$ on $U\setminus V$. By construction, this diffeomorphism satisfies
$$\begin{array}{rcll}
\bar \Phi_{in} & = &  \mathrm{Id} & \mbox{ on }\partial^{in} U\\
& = & \psi_2^{-1}\circ\psi_1 & \mbox{ on }\partial^{out} U.
\end{array}$$
As a consequence, the relation $\pi_2\circ \bar \Phi_{in}=\bar \Phi_{in}\circ \pi_1$ holds, and therefore $\bar\Phi_{in}$ induces a diffeomorphism $\Phi_{in}:M_1\to M_2$. Since $\bar \Phi_{in}=  \mathrm{Id}$  on $\partial^{in} U$, it follows that $\Phi_{in|S_1}=\pi_{2|\partial^{in} U}\circ (\pi_{2|\partial^{in} U})^{-1}=\phi_{in}$, as desired.
\end{proof}

Now, we introduce the return maps on the surface $S_1$ and $S_2$. We first consider the crossing map of the plug $(U,X)$
$$\theta_X:\partial^{in} U\setminus L^s\to\partial^{out} U\setminus L^u$$ 
By definition, $\theta_X(x)$ is the unique intersection point of the forward $(X^t)$-orbit of the point $x$ with the surface $\partial^{out} U$. For $i=1,2$, the map $\theta_X$ induces a map
$$\theta_i:=\pi_{i|\partial^{out} U}\circ \theta_X\circ (\pi_{i|\partial^{in} U})^{-1}:S_i\setminus L^s_i\to S_i\setminus L^u_i.$$
This map $\theta_i$ is just the Poincar\'e return map of the flow $(Y_i^t)$ on the surface $S_i$. 

\begin{proposition}
\label{p.relation}
The diffeomorphisms $\theta_1$, $\theta_2$, $\phi_{in}$ and $\phi_{out}$ are related by the following equality 
$$\theta_2\circ \phi_{in}=\phi_{out}\circ\theta_1.$$
\end{proposition}

\begin{proof}
This is an immediate consequence of the formulas defining $\theta_1$, $\theta_2$, $\phi_{in}$ and $\phi_{out}$.
\end{proof}

Now we lift all the objects to the universal covers of $M_1$ and $M_2$. We pick a point $x_1\in M_1$ which will serve as the base point of the fundamental group of the manifold $M_1$. The point $x_2:=f(x_1)$ will be used as the base point of fundamental group of the manifold $M_2$. The diffeomorphism $\Phi_{in}$ provides us with an isomorphism $(\Phi_{in})_*$ between the fundamental groups $\pi_1(M_1,x_1)$ and $\pi_1(M_2,x_2)$. For $i=1,2$, we denote by $p_i:\widetilde M_i\to M_i$ the universal cover of the manifold $M_i$.  We denote by $\widetilde Y_i$ the lift of the vector field $Y_i$ on $\widetilde M_i$. Observe that $\widetilde Y_i$ is equivariant under the action of $\pi_1(M_i,x_i)$: for $\gamma\in\pi_1(M_i,x_i)$, one has $\widetilde Y_i(\gamma \tilde x)=D_{\tilde x}\gamma.\widetilde Y_i(\tilde x)$. We denote by $\widetilde S_i$ the complete lift of the surface $S_i$ (\emph{i.e.} $\widetilde S_i:=p_i^{-1}(S_i)$). 

We denote by $\widetilde L^s_i$ and $\widetilde L^u_i$ the complete lifts of the laminations $L^s_i$ and $L^u_i$. We denote by 
$$\widetilde\theta_i:\widetilde S_i\setminus\widetilde L^s_i\to\widetilde S_i\setminus L^u_i$$
the first return map of the flow of the vector field $\widetilde Y_i$ on the surface $\widetilde S_i$. Clearly, $\widetilde\theta_i$ is a lift of the map $\theta_i$. Moreover, $\widetilde\theta_i$ commutes with the deck transformations:
\begin{equation}
\label{e.lift-theta-deck}
\tilde\theta_i\circ\gamma=\gamma\circ\tilde\theta_i \mbox{ for every }\gamma\in\pi_1(M_i,x_i).
\end{equation}
This commutation relation is an immediate consequence of the equivariance of $\widetilde Y_i$ (see above).
Now we fix a lift $\widetilde \Phi_{in}:\widetilde M_1\to\widetilde M_2$ of the diffeomorphism $\Phi_{in}$ (note that, unlike what happens for $\theta_1$ and $\theta_2$, there is no canonical lift of $\Phi_{in}$). Recall that the diffeomorphism $\Phi_{in}$ maps the surface $S_1$ to the surface $S_2$, and that the restriction of $\Phi_{in}$ to $S_1$ coincides with $\phi_{in}$. As a consequence, the lift $\tilde \Phi_{in}$ maps the surface $\widetilde S_1$ to $\widetilde S_2$, and the restriction of $\tilde \Phi_{in}$ to $\widetilde S_1$ is a lift $\tilde \phi_{in}$ of the diffeomorphism $\phi_{in}$. By construction, this lift satisfies 
\begin{equation}
\label{e.lift-phi-s-deck}
\tilde\phi_{in}\circ\gamma=(\Phi_{in})_*(\gamma)\circ\tilde\phi_{in} \mbox{ for every }\gamma\in\pi_1(M_1,x_1)
\end{equation}
Now recall that, according to Proposition~\ref{p.isotopy}, there exists a continuous arc $(\phi_t)_{t\in [0,1]}$ of diffeomorphisms from $S_1$ to $S_2$, such that $\phi_0=\phi_{in}$ and $\phi_1=\phi_{out}$, and such that the laminations $\phi_t(L^u_1)$ and $L^s_2$ are strongly transverse for every $t$. We lift this isotopy, starting at the lift $\tilde\phi_{in}$ of $\phi_{in}=\phi_0$. This yields a continuous arc $(\tilde\phi_t)_{t\in [0,1]}$ of diffeomorphisms from $\widetilde S_1$ to $\widetilde S_2$, such that $\tilde\phi_0=\tilde\phi_{in}$ and such that the laminations $\tilde\phi_t(\widetilde L^u_1)$ and $\widetilde L^s_2$ are strongly transverse for every $t$. The difffeomorphism $\tilde\phi_{out}:=\tilde\phi_1$ is a lift of the diffeomorphism $\phi_{out}$. By continuity, the relation~\eqref{e.lift-phi-s-deck} remains true if we replace $\tilde\phi_{in}=\tilde\phi_0$ by $\tilde\phi_t$ for any $t\in [0,1]$. In particular, the diffeomorphism $\tilde\phi_{out}$ satisfies
\begin{equation}
\label{e.lift-phi-u-deck}
\tilde\phi_{out}\circ\gamma=(\Phi_{in})_*(\gamma)\circ\tilde\phi_{out} \mbox{ for every }\gamma\in\pi_1(M_1,x_1).
\end{equation}

\begin{proposition}
\label{p.relation-lifts}
The diffeomorphisms $\tilde\theta_1$, $\tilde\theta_2$, $\tilde\phi_{in}$ and $\tilde\phi_{out}$ are related by the equality
$$\tilde\theta_2\circ\tilde\phi_{in}=\tilde\phi_{out}\circ\tilde\theta_1.$$
\end{proposition}

\begin{proof}
According to Proposition~\ref{p.relation}, the diffeomorphisms $\theta_2\circ \phi_{in}$ and $\phi_{out}\circ\theta_1$ coincide. Hence the diffeomorphisms $\tilde\theta_2\circ\tilde\phi_{in}$ and $\tilde\phi_{out}\circ\tilde\theta_1$ are two lifts of the same diffeomorphism. It follows that there exists a deck transformation $\gamma_0\in\pi_1(M_2,y_0)$ such that 
$$\tilde\theta_2\circ\tilde\phi_{in}=\gamma_0\circ\tilde\phi_{out}\circ\tilde\theta_1.$$
Now consider a deck transformation $\gamma\in \pi_1(M_1,x_0)$. On the one hand, using~\eqref{e.lift-phi-s-deck} and~\eqref{e.lift-theta-deck}, we get 
$$\tilde\theta_2\circ\tilde\phi_{in}\gamma=\tilde\theta_2\circ (\Phi_{in})_*(\gamma)\circ \tilde\phi_{in}=(\Phi_{in})_*(\gamma)\circ \tilde\theta_2\circ\tilde\phi_{in}=\left((\Phi_{in})_*(\gamma)\cdot\gamma_0\right)\circ\tilde\phi_{out}\circ\tilde\theta_1.$$
On the other hand, using~\eqref{e.lift-theta-deck} and~\eqref{e.lift-phi-u-deck}, we get
$$\tilde\theta_2\circ\tilde\phi_{in}\circ\gamma=\gamma_0\circ\tilde\phi_{out}\circ\tilde\theta_1\circ\gamma=\gamma_0\circ\tilde\phi_{out}\circ\gamma\circ\tilde\theta_1=\left(\gamma_0\cdot(\Phi_{in})_*(\gamma)\right)\circ\tilde\phi_{out}\circ\tilde\theta_1.$$
Hence
$$(\Phi_{in})_*(\gamma)\cdot\gamma_0=\gamma_0\cdot(\Phi_{in})_*(\gamma).$$
Since $(\Phi_{in})_*(\gamma)$ ranges over the whole fundamental group $\pi_1(M_2,y_0)$, it follows that $\gamma_0$ is in the center of the fundamental group $\pi_1(M_2,y_0)$. If $\gamma_0\neq\mathrm{Id}$, this implies that $\pi_1(M_2,y_0)$ has a non-trivial center. Then, a (easy generalization of) well-known theorem of \'E. Ghys implies that, up to finite cover, the Anosov flow $(X_2^t)$ must be topologically equivalent to the geodesic flow on the unit tangent bundle of a closed hyperbolic surface  (see~\cite{Ghy}, or~\cite[Th\'eor\`eme 3.1]{Bar06}). This is clearly impossible, since $X_2$ admits a transverse torus (any connected component of the surface $S_2$ is such a torus). As a consequence, $\gamma_0$ must be the identity, and the desired relation $\tilde\theta_2\circ\tilde\phi_{in}=\tilde\phi_{in}\circ\tilde\theta_1$ is proved.
\end{proof}

\subsection{Construction of maps $\Delta^s:f_1^{s,\infty}\to f_2^{s,\infty}$ and $\Delta^u:f_1^{u,\infty}\to f_2^{u,\infty}$}

In section~\ref{s.coding}, we have defined some symbolic spaces which allow to code certain orbits of certain Anosov flows. Let us introduce these symbolic space in our particular setting. For $i=1,2$, we consider the alphabet 
$$\cA_i:=\{\mbox{ connected components of }\widetilde S_i\setminus\widetilde L^s_i\},$$
and the symbolic space
$$\Sigma_i:=\{(D_{p})_{p\in\ZZ}\mbox{ such that }D_{p} \in \cA_i \mbox{ and }\tilde\theta_i(D_{p})\cap D_{p+1}\neq\emptyset\mbox{ for every }p\}.$$ 
In order to code stable and unstable leaves, we consider the subspaces $\Sigma_i^s$ and $\Sigma_i^u$ of $\Sigma_i$ defined by 
$$\Sigma_i^s:=\{(D_{p})_{p\geq 0}\mbox{ such that }D_p \in \cA_i \mbox{ and }\tilde\theta_i(D_p)\cap D_{p+1}\neq\emptyset\mbox{ for every }p\}$$ 
and 
$$\Sigma_i^u:=\{(D_p)_{p<0}\mbox{ such that }D_p \in \cA_i \mbox{ and }\tilde\theta_i(D_p)\cap D_{p+1}\neq\emptyset\mbox{ for every }p\}.$$ 

\begin{proposition}
\label{p.equivalence-transitions}
Let $D_1$ and $D_1'$ be two elements of $\cA_1$. Let $D_2:=\widetilde\phi_{in}(D_1)$ and $D_2':=\widetilde\phi_{in}(D_1')$. Then $\widetilde\theta_1(D_1)$ intersects $D_1'$ if and only if $\widetilde\theta_2(D_2)$ intersects $D_2'$. 
\end{proposition}

\begin{proof}
We have the following sequence of equivalences.
\begin{eqnarray*}
\widetilde\theta_1(D_1)\cap D_1'  \neq  \emptyset & \displaystyle{\mathop{\Longleftrightarrow}^{(1)}} & \widetilde\phi_{in}(\widetilde\theta_1(D_1))\cap\widetilde\phi_{in}(D_1')  \neq  \emptyset \\
&  \displaystyle{\mathop{\Longleftrightarrow}^{(2)}} & \widetilde\phi_{out}(\widetilde\theta_1(D_1))\cap\widetilde\phi_{in}(D_1')  \neq  \emptyset  \\
& \displaystyle{\mathop{\Longleftrightarrow}^{(3)}} & \widetilde\theta_2(\widetilde\phi_{in}(D_1))\cap\widetilde\phi_{in}(D_1')  \neq \emptyset \\
& \displaystyle{\mathop{\Longleftrightarrow}^{(4)}} & \widetilde\theta_2(D_2)\cap D_2'  \neq  \emptyset.
\end{eqnarray*}
The first equivalence is straightforward. The last one is nothing but the definition of the connected components $D_2$ and $D_2'$. Equivalence $(3)$ follows from Proposition~\ref{p.relation-lifts}. It remains to prove equivalence $(2)$. For that purpose, observe that $\widetilde\theta_1(D_1)$ is a strip bounded by two leaves of $\widetilde L^u_1$, and $\widetilde\phi_{in}(D_1')$ is a strip bounded by two leaves of $\widetilde L^s_2$. Now recall that there exists an isotopy $(\widetilde\phi_t)_{t\in [0,1]}$ joining $\widetilde\phi_{in}$ to $\widetilde\phi_{out}$, such that the lamination $\widetilde\phi_t(\widetilde L^u_1)$ is  strongly transverse to the lamination $\widetilde L^s_2$. It follows that $\widetilde\phi_{out}(\widetilde\theta_1(D_1))$ intersects $\widetilde\phi_{in}(D_1')$ if and only if $\widetilde\phi_{in}(\widetilde\theta_1(D_1))$ intersects $\widetilde\phi_{in}(D_1')$. 
\end{proof}

As an immediate consequence of Proposition~\ref{p.equivalence-transitions}, we get:

\begin{corollary}
\label{c.symbolic-space-to-symbolic-space}
$(\tilde\phi_{in})^{\otimes\ZZ}:\cA^{\ZZ}\to\cA^{\ZZ}$ maps $\Sigma_1$ to $\Sigma_2$. 
\end{corollary}

Corollary~\ref{c.symbolic-space-to-symbolic-space} entails that $(\tilde\phi_{in})^{\otimes\ZZ_{\geq 0}}$ maps $\Sigma_1^s$ to $\Sigma_2^s$, and $(\widetilde\phi_{in})^{\otimes\ZZ_{<0}}$ maps $\Sigma_1^u$ to $\Sigma_2^u$. Hence, the map $\tilde\phi_{in}$ builds a bridge between the symbolic spaces associated to the vector filed $Y_1$ and those associated to the vector filed $Y_2$.

Let us recall the definition of the coding maps constructed in Section~\ref{ss.coding-procedure}. For $i=1,2$, we denote by $\cF^s_i$ and $\cF^u_i$ the weak stable and the weak unstable foliations of the Anosov flow $(Y_i^t)$ on the manifold $M_i$. These two-dimensional foliations induce two one-dimensional foliations $F^s_i$ and $F^u_i$ on the surface $S_i$. We denote by $\widetilde F^s_i$ and $\widetilde F^u_i$ the lifts of $F^s_i$ and $F^u_i$ on $\widetilde S_i$. We denote by $f^s_i$ and $f^u_i$ the leaf spaces of the foliations $\widetilde F^s_i$ and $\widetilde F^u_i$. We denote by $f^{s,\infty}_i$ the subset of $f^s_i$ made of the leaves that are not in $\Wi W^s(\Lambda_i)$ (recall that $\Wi W^s(\Lambda_i)$ is a union of leaves of $\Wi\cF^s_i$ and therefore $\Wi W^s(\Lambda_i)\cap\Wi S_i$ is a union of leaves of $F^s_i$). Similarly, we  denote by $f^{u,\infty}_i$ the subset of $f^u_i$ made of the leaves that are not in $\Wi W^u(\Lambda_i)$. The construction of subsection~\ref{ss.coding-procedure} provides two bijective coding maps
$$\begin{array}{crcll}
\chi_i^s :&\tilde f^{s,\infty}_i & \longrightarrow & \Sigma^s_i \\
& \ell & \longmapsto &  (D_p)_{p\geq 0} &\mbox{where }\tilde\theta_i^p (\ell) \subset D_p\mbox{ for every }p\geq 0\\
\end{array}$$
and 
$$\begin{array}{crcll}
\chi_i^u :& \tilde f^{u,\infty}_i  & \longrightarrow & \Sigma^u_i\\
& \ell & \longmapsto & (D_p)_{p<0} &\mbox{where }\tilde\theta_i^p (\ell) \subset D_p\mbox{ for every }p< 0
\end{array}$$

Hence we obtain two natural bijective maps  
$$\Delta^s:=(\chi^s_2)^{-1}\circ(\tilde\phi_{in})^{\otimes\ZZ_{\geq 0}}\circ\chi^s_1 : \tilde f^{s,\infty}_1\longrightarrow \tilde f^{s,\infty}_2$$
and
$$\Delta^u:=(\chi^u_2)^{-1}\circ(\tilde\phi_{in})^{\otimes\ZZ_{< 0}}\circ\chi^u_1 : \tilde f^{u,\infty}_1\longrightarrow \tilde f^{u,\infty}_2$$

\subsection{Extension of the maps $\Delta^s$ and $\Delta^u$}

We wish to extend the map $\Delta^s$ in order to obtain a bijective map between the leaf spaces $\tilde f^s_1$ and $\tilde f^s_2$. In view to that goal, we will prove that $\Delta^s$ preserves the order of the leaves of the foliations $\widetilde F^s_1$ and $\widetilde F^s_2$.  Our first task is to write a precise definition of these order. First we choose an orientation of the lamination $L^u_X\subset\partial^{out} U$. Pushing this orientation by the maps $\pi_1$ and $\pi_2$, this defines some orientations of the laminations $L^u_1=(\pi_1)_*(L^u_X)\subset S_1$ and $L^u_2=(\pi_2)_*(L^u_X)\subset S_1$. Since $L^u_i$ is a sublamination of the foliation $F^u_i$ (and since $L^u_i$ intersects every connected component of $S_i$), the orientations of the lamination $L^u_1$ and $L^u_2$ define some orientations of the foliations $F^u_1$ and $F^u_2$. Finally, these orientation can be lifted, providing orientations of the lifted foliations $\widetilde F^u_1$ and $\widetilde F^u_2$. It is important to notice that our choice of orientations for $\widetilde F^u_1$ and $\widetilde F^u_2$ are not independent from each other. More precisely, the orientation are chosen so that $\phi_{out}=\pi_{2|\partial^{out} U}\circ(\pi_{2|\partial^{out} U})^{-1}$ maps the orientation of the lamination $L^u_1$ to the orientation of the lamination $L^u_2$, and therefore: 
\begin{equation}
\label{e.phi-out-preserves-orientation}
\tilde\phi_{out}\mbox{ maps the orientated lamination }\widetilde L^u_1\mbox{ to the orientated lamination }\widetilde L^u_2.
\end{equation}
As explained in Subsection~\ref{ss.partial-orders}, the orientation of the foliation $\widetilde F^u_i$ induces a partial order $\prec_i$ on the leaf space $\tilde f^s_i$ defined as follows: given two leaves $\ell_i,\ell_i'\in \tilde f^s_i$ satisfy $\ell_i\prec_i \ell_i'$ if there exists an arc segment of an oriented leaf of $\widetilde F^u_i$ going from a point of $\ell_i$ to a point of $\ell_i'$. Proposition~\ref{p.pre-order} proves that this indeed defines an order on $\tilde f^s_i$. Moreover, this order on $\tilde f^s_i$ induces a partial order on the alphabet $\cA_i$: given two elements $D_i$ and $D_i'$ of $\cA_i$, we write $D_i\prec_i D_i'$ if there exists a leaf $\tilde \alpha_i$ of $\widetilde F^s_i$ included in $D_i$ and  a leaf $\tilde \alpha_i'$ of $\widetilde F^s_i$ included in $D_i'$  such that $\tilde\alpha_i\prec_i\tilde\alpha_i'$. Proposition~\ref{p.order-compatible-cc} shows that we can replace ``there exists'' by "for every" in this definition. It follows that $\prec_i$ is indeed a partial order on $\cA_i$. Now comes the technical result which will allow us to extend the map $\Delta^s$:

\begin{proposition}
\label{p.Delta-preserves-order}
The map $\Delta^s:(f^{s,\infty}_1,\prec_1)\longrightarrow (f^{s,\infty}_2,\prec_2)$ is order-preserving. 
\end{proposition}

In order to prove Proposition~\ref{p.Delta-preserves-order}, we need several intermediary results.

\begin{lemma}
\label{l.phi-preserves-order}
The map $\tilde\phi_{in}:(\cA_1,\prec_1)\longrightarrow (\cA_2,\prec_2)$ is order-preserving. 
\end{lemma}

\begin{proof}
Consider two elements $D_1,D_1'$ of $\cA_1$. Assume that $D_1\prec_1 D_1'$. This means that there exists a leaf $\ell_1$ of the oriented lamination $\widetilde L^u_1$ which crosses $D_1$ before crossing $D_1'$. As a consequence, if we endow $\tilde\phi_{in}(\ell_1)$ with the image under $\tilde\phi_{in}$ of the orientation of $\alpha_1$, then $\tilde\phi_{in}(\ell_1)$ crosses $\tilde\phi_{in}(D_1)$ before crossing $\tilde\phi_{in}(D_1')$. Now recall that:
\begin{itemize}
\item $\tilde\phi_{in}(D_1)$ and $\tilde\phi_{in}(D_1')$ are strips bounded by leaves of the lamination $\tilde\phi_{in}(\widetilde L^s_1)=\widetilde L^s_2$,
\item there exists an isotopy $(\tilde\phi_t)$ joining $\tilde\phi_{in}$ to $\tilde\phi_{out}$ such that the lamination $\tilde\phi_t(\widetilde L^u_1)$ is strongly transverse to the lamination $\widetilde L^s_2$ for every $t$.
\end{itemize}
We deduce that, if we endow  $\tilde\phi_{out}(\ell_1)$ with the image under $\tilde\phi_{out}$ of the orientation of $\ell_1$, then $\tilde \phi_{out}(\ell_1)$ crosses $\tilde\phi_{in}(D_1)$ before crossing $\tilde\phi_{in}(D_1')$. According to~\eqref{e.phi-out-preserves-orientation}, this means that there is a leaf of the oriented lamination $\widetilde L^u_1$ which crosses $\tilde\phi_{in}(D_1)$ before crossing $\tilde\phi_{in}(D_1')$. By definition of the partial order $\prec_2$, this means that  $\tilde\phi_{in}(D_1)\prec_2\tilde\phi_{in}(D_1')$.
\end{proof}

\begin{lemma}
\label{l.equivalence-orientation}
Let $D_1$ be a connected component of $\widetilde S_1\setminus\widetilde L^s_1$. Set $D_2:=\tilde\phi_{in}(D_1)$. Then  the following are equivalent:
\begin{enumerate}
\item the map $\tilde\theta_1$ restricted to the strip $D_1$ preserves the orientation of the foliation $\widetilde F^u_1$,
\item the map $\tilde\theta_2$  restricted to the strip $D_2$ preserves the orientation of the foliation $\widetilde F^u_2$.
\end{enumerate}
\end{lemma}

\begin{proof}
The proof is a bit intricate, because we need to introduce no less than six leaves and compare their orientations. Recall that we have chosen some orientations for the foliations $\widetilde F^u_1$ and $\widetilde F^u_2$. In the sequel, we will also consider the foliations $(\tilde\phi_{in})_*\widetilde F^u_1$, $(\tilde\phi_{out})_*\widetilde F^u_1$ and $(\tilde\phi_{t})_*\widetilde F^u_1$ ; we endow them with the images under $\tilde\phi_{in}$, $\tilde\phi_{out}$ and $\tilde\phi_{t}$ of the orientation of $\widetilde F^u_1$. 

We pick a leaf $\ell_1$ of the lamination $\widetilde L^u_1$ so that $\ell_1\cap D_1\neq\emptyset$ (such a leaf always exists since the laminations $\widetilde L^s_1$ and $\widetilde L^u_1$ are strongly transverse). Then we set
$$\ell_2:=\tilde\phi_{out}(\ell_1)\quad\quad\hat\ell_2:=\tilde\phi_{in}(\ell_1)\quad\quad\ell_1':=\tilde\theta_{1}(\ell_1\cap D_1)\quad\quad\ell_2':=\tilde\theta_{2}(\ell_2\cap D_2)\quad\quad\hat\ell_2':=\tilde\theta_{2}(\hat\ell_2\cap D_2).$$
Observe that 
\begin{equation}
\label{e.hat-ell_2'}
\hat\ell_2'=\tilde\theta_{2}(\tilde\phi_{in}(\ell_1)\cap D_2)=\tilde\theta_{2}\circ\tilde\phi_{in}(\ell_1\cap D_1)=\tilde\phi_{out}\circ \tilde\theta_{1}(\ell_1\cap D_1)=\tilde\phi_{out}(\ell_1')
\end{equation}
(the third equality follows from Proposition~\ref{p.relation-lifts}).  Now recall that, for $i=1,2$, both $\widetilde L^u_i$ and $(\tilde\theta_i)_*(\widetilde L^u_i\cap D_i^s)$ are sublaminations of the foliation $\widetilde\cF^u_i$. Also recall that $\tilde\phi_{out}(\widetilde L^u_1)=\widetilde L^u_2$. This provides some natural orientations on $\ell_1,\ell_1',\ell_2,\ell_2',\hat\ell_2,\hat\ell_2'$:
\begin{itemize}
\item $\ell_1$ and $\ell_1'$ are leaves of the foliation $\widetilde F^u_1$, hence inherit of the orientation of $\widetilde F^u_1$; 
\item $\ell_2$ and $\ell_2'$ are leaves of the foliation $\widetilde F^u_2$, hence inherit of the orientation of $\widetilde F^u_2$; we endow them with the orientation of this foliation;
\item $\hat\ell_2$ is a leaf of the foliation $(\tilde\phi_{in})_*\widetilde F^u_1$,  hence inherits of the orientation of $(\tilde\phi_{in})_*\widetilde F^u_1$;
\item $\hat\ell_2'$ is a leaf of the foliation $(\tilde\phi_{out})_*\widetilde F^u_1$, hence inherits of the orientation of $(\tilde\phi_{out})_*\widetilde F^u_1$.
\end{itemize}

By symmetry, it is enough to prove the implication $1\Rightarrow 2$. So, we assume that the restriction of $\tilde\theta_1$ to $D_1^s$ preserves the orientation of $\widetilde F^u_1$;  in particular:
\begin{equation}
\label{e.1}
\mbox{$\tilde\theta_1$ maps the orientation of $\ell_1$ to those of $\ell_1'$.}
\end{equation}
According to~\eqref{e.phi-out-preserves-orientation},
\begin{equation}
\label{e.2}
\mbox{$\tilde\phi_{out}$ maps the orientation of $\ell_1$ to those of $\ell_2$.}
\end{equation}
The orientations of $\ell_1,\ell_2,\hat\ell_2,\hat\ell_2'$ are chosen in such a way that $\tilde\phi_{in}^{-1}$ maps the orientation of $\hat\ell_2$ to those of $\ell_1$, and $\tilde\phi_{out}$ maps the orientation of $\ell_1'$ to those of $\hat\ell_2'$. Puting this together with~\eqref{e.1}, we obtain that $\tilde\phi_{out}\circ\tilde\theta_1\circ\tilde\phi_{in}^{-1}$ maps the orientation of $\hat\ell_2$ to those of~$\hat\ell_2'$. Using proposition~\ref{p.relation-lifts}, we obtain
\begin{equation}
\label{e.3}
\mbox{$\tilde\theta_2$ maps the orientation of $\hat\ell_2$ to those of $\hat\ell_2'$.}
\end{equation}

Our final goal is to prove that $\tilde\theta_2$ maps the orientation of $\ell_2$ to those of $\ell_2'$. So, in view of~\eqref{e.3}, we need to compare the orientations of $\ell_2$ and $\hat\ell_2$ on the one hand, and the orientations $\ell_2'$ and $\hat\ell_2'$ on the other hand. We start by $\ell_2$ and $\hat\ell_2$.

Recall that $D_2$ is a strip in $\widetilde S_2$ bounded by two leaves of the stable lamination $\widetilde L^s_2$. We denote these two leaves by $\alpha$ and $\beta$, in such a way that oriented unstable leaf $\ell_2$ enters in $D_2$ by crossing $\alpha$ and exits $D_2$ by crossing $\beta$. According to~\eqref{e.2}, the orientation of $\ell_2=(\tilde\phi_{out})_*\ell_1$ as a leaf of $\widetilde L^u_2\subset\widetilde F^u_2$ coincides with the orientation as a leaf of $(\tilde\phi_{out})_*\widetilde L^u_1\subset (\tilde\phi_{out})_* F^u_1$. Moreover, recall that there exists an isotopy $(\tilde\phi_t)_{t\in [0,1]}$ joining $\tilde\phi_0=\tilde\phi_{in}$ to  $\tilde\phi_1=\tilde\phi_{out}$, such that the lamination $\tilde\phi_t(\widetilde L^u_1)$ is strongly transverse to the lamination $\widetilde L^u_2$ for every $t$. We deduce that $\hat\ell_2=(\tilde\phi_{in})_*(\ell_1)$ crosses $D_2$ in the same direction as $\ell_2=(\tilde\phi_{out})_*\ell_1$. In other words,  
\begin{equation}
\label{e.4}
\mbox{both $\ell_2$ and $\hat\ell_2$ enter in $D_2$ by crossing $\alpha$ and exits $D_2$ by crossing $\beta$.}
\end{equation}
Let $U$ and $V$ be some disjoint neighborhoods of the stable leaves $\alpha$ and $\beta$ in the strip $D_2$. Assertion~\eqref{e.4} can be reformulated as follows
\begin{equation}
\label{e.5}
\mbox{the arcs of oriented leaves $\ell_2\cap D_2$ and $\hat\ell_2\cap D_2$ both go from $U$ to $V$.}
\end{equation}

We 	are left to compare the orientations of $\ell_2'$ and $\hat\ell_2'$.  First observe that $\tilde\theta_2(D_2)$ is an open strip in $\widetilde S_2$, bounded by two leaves of the unstable lamination $\widetilde L^u_2=(\tilde\phi_{out})_*\widetilde L^u_1$. The closure $\mbox{Cl}(\tilde\theta_2(D_2))$ of $\tilde\theta_2(D_2)$ is the union of the open strip $\tilde\theta_2(D_2)$ and its two boundary leaves. The boundary components of $\tilde\theta_2(D_2)$ are leaves of both the foliations $F^u_2$ and $(\tilde\phi_{out})_*F^u_1$. Moreover, $F^u_2$ and $(\tilde\phi_{out})_*F^u_1$ induce two trivial oriented foliations on the closed strip  $\mbox{Cl}(\tilde\theta_2(D_2))$. In particular, the leaves of $F^u_2$ and $(\tilde\phi_{out})_*F^u_1$ in $\mbox{Cl}(\tilde\theta_2(D_2))$ go from one end of $\mbox{Cl}(\tilde\theta_2(D_2))$ to the other end. In order to distinguish the two ends of the closed strip $\mbox{Cl}(\tilde\theta_2(D_2))$, we use the set $\mbox{Cl}(\tilde\theta_2(U))$ and  $\mbox{Cl}(\tilde\theta_2(V))$. These sets are disjoint neighbourhoods of the two ends of $\mbox{Cl}(\tilde\theta_2(D_2))$. So we just need to decide if the leaves go from $\mbox{Cl}(\tilde\theta_2(U))$ and  $\mbox{Cl}(\tilde\theta_2(V))$, or the contrary. On the one hand, putting~\eqref{e.3} and~\eqref{e.5} together, we obtain that $\hat\ell_2$ goes from $\mbox{Cl}(\tilde\theta_2(U))$ to $\mbox{Cl}(\tilde\theta_2(V))$. On the other hand,  $F^u_2$ and $(\tilde\phi_{out})_*F^u_1$ are trivial oriented foliations on $\mbox{Cl}(\tilde\theta_2(D_2))$, and, according to~\eqref{e.phi-out-preserves-orientation}, they induce the same orientation on the boundary leaves of $D_2'$. So we conclude that all the leaves of both the oriented foliations $F^u_2$ and $(\tilde\phi_{out})_*F^u_1$ go from $\mbox{Cl}(\tilde\theta_2(U))$ to $\mbox{Cl}(\tilde\theta_2(V))$. In particular, 
\begin{equation}
\label{e.6}
\mbox{the oriented leaves $\ell_2'$ and $\hat\ell_2'$ go from $\tilde\theta_2(U)$ to $\tilde\theta_2(V)$.}
\end{equation}
From~\eqref{e.5} and~\eqref{e.6}, we deduce that $\tilde\theta_{2|D_2}$ maps the orientation of $\ell_2$ to those of $\ell_2'$. By definition of the orientations of  $\ell_2$ and $\ell_2'$, this means that the restriction of  $\tilde\theta_2$ to the strip $D_2$ preserves the orientation of the foliation $\widetilde F^u_2$. This completes the proof of the implication $1\Rightarrow 2$. 
\end{proof}

\begin{figure}[ht]
\begin{center}
  \includegraphics[totalheight=12cm]{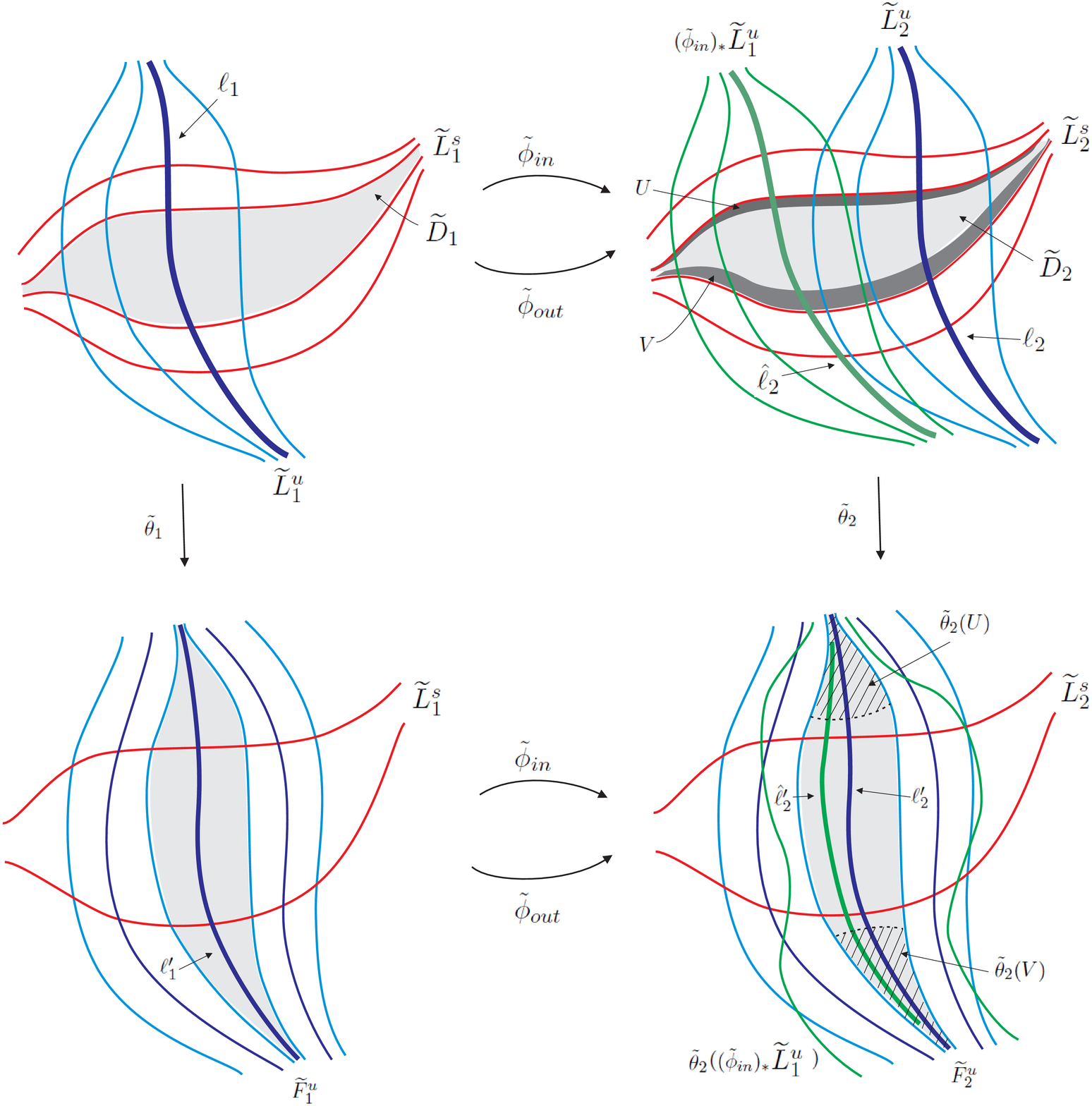}
\caption{\label{f.orientation}Proof of Lemma~\ref{l.equivalence-orientation}}
 \end{center}
\end{figure}

\begin{corollary}
\label{c.equivalence-orientation}
Let $D_{1,0},\dots,D_{1,p_0-1}$ be connected components of $\widetilde S_1\setminus\widetilde L^s_1$, so that $\bigcap_{p=0}^{p_0-1}\widetilde\theta_1^p(D_{1,p})$ is non-empty. For $p=1,\dots,p_0-1$, let $D_{2,p}:=\tilde\phi_{in}( D_{1,p})$. Then the following are equivalent:
\begin{enumerate}
\item the map $\widetilde\theta_1^{p_0}$ restricted to $\bigcap_{p=0}^{p_0-1}\widetilde\theta_1^p(D_{1,p})$ preserves the orientation of the foliation $\widetilde F^u_1$,
\item the map $\widetilde\theta_2^{p_0}$ restricted to $\bigcap_{p=0}^{p_0-1}\widetilde\theta_2^p(D_{2,p})$ preserves the orientation of the foliation $\widetilde F^u_2$.
\end{enumerate}
\end{corollary}

\begin{proof}
For $i=1,2$, consider the set 
$$J_i:=\left\{j\in\{0,\dots,p_0-1\}\mbox{ s. t. the restriction of }\widetilde\theta_i\mbox{ to }D_{i,p}\mbox{ preserves the orientation of }\widetilde F^u_i\right\}.$$ 
On the one hand, Lemma~\ref{l.equivalence-orientation} implies that the sets $J_1$ and $J_2$ coincide. On the other hand, it is clearly that the restriction of $\widetilde\theta_i$ to $\bigcap_{j=0}^{p_0-1}\widetilde\theta_i^p(D_{i,p})$ preserves the orientation of the leaves of $\widetilde F^u_i$ if and only if the cardinality of $J_i$ is even. The corollary follows. 
 \end{proof}

\begin{proof}[Proof of Proposition~\ref{p.Delta-preserves-order}]
We consider two leaves $\gamma_1$ and $\gamma_1'$ in $f^{s,\infty}_1$, we denote $\gamma_2:=\Delta^s(\gamma_1)$ and $\gamma_2':=\Delta^s(\gamma_1')$, and we assume that $\gamma_1\prec_{1}\gamma_1'$. We aim to prove $\gamma_2\prec_{2} \gamma_2'$. Let 
$$\chi_1^s(\tilde\gamma_1)=\left(D_{1,p}\right)_{p\geq 0}\quad\quad\chi_1^s(\tilde\gamma_1')=\left(D_{1,p}'\right)_{p\geq 0}\quad\quad\chi_2^s(\tilde\gamma_2)=\left(D_{2,p}\right)_{p\geq 0}\quad\quad\chi_2^s(\tilde\gamma_2')=\left(D_{2,p}'\right)_{p\geq 0}.$$
By defintion of the map $\chi_i^s$, this means that, for $i=1,2$, 
$$\tilde\gamma_i=\bigcap_{p\geq 0}\tilde\theta_i^{-p}(D_{i,p}) \quad\mbox{ and  }\quad\tilde\gamma_i'=\bigcap_{p\geq 0}\tilde\theta_i^{-p}(D_{i,p}').$$
And since $\tilde\gamma_2=\Delta^s(\tilde\gamma_1)$ and $\tilde\gamma_2'=\Delta^s(\tilde\gamma_1')$, we have 
$$D_{2,p}=\phi_{in}\left(D_{1,p}\right)\quad\mbox{and}\quad D_{2,p}'=\phi_{in}\left(D_{1,p}'\right)$$
for every $p\geq 0$. We denote by $p_0$ the smallest integer $p$ such that $D_{1,p}\neq D_{1,p}'$. 

Let us consider the case where the map $\widetilde\theta_1^{p_0}$ restricted to $\bigcap_{p=0}^{p_0-1}\tilde\theta_1^{-p}(D_{1,p})$ preserves the orientation of the foliation $\widetilde F^u_1$.
\begin{itemize}
\item Proposition~\ref{p.oricoh} implies that $D_{1,p_0}\prec_1 D_{1,p_0}'$.
\item Since $\phi_{in}:\cA_1\to\cA_2$ is order-preserving (Lemma~\ref{l.phi-preserves-order}), it follows that $D_{2,p_0}\prec_2 D_{2,p_0}'$.
\item Corollary~\ref{c.equivalence-orientation} implies that the map  $\widetilde\theta_2^{p_0}$, restricted to $\bigcap_{p=0}^{p_0-1}\tilde\theta_2^{-p}(D_{2,p})$ preserves the orientation of the foliation $\widetilde F^u_2$.
\item Using again Proposition~\ref{p.oricoh}, we deduce from the two last items above that $\tilde\gamma_2\prec_2\tilde\gamma_2'$, as desired.
\end{itemize}
The case where the map $\widetilde\theta_1^{p_0}$ restricted to $\bigcap_{p=0}^{p_0-1}\tilde\theta_1^{-p}(D_{1,p})$ reverses the orientation of the foliation $\widetilde F^u_1$ follows from the very same arguments. 
\end{proof}

\begin{corollary}
\label{c.extension-s}
The map $\Delta^s: f^{s,\infty}_1\longrightarrow f^{s,\infty}_2$ extends in a unique way to an order-preserving bijection $\Delta^s: f^{s}_1\longrightarrow f^{s}_2$.
\end{corollary}

\begin{proof}
%
%
%
%
This is an immediate consequence of the following facts:
\begin{itemize}
\item $\Delta_s: f^{s,\infty}_1\longrightarrow f^{s,\infty}_2$ is an order-preserving map (Proposition~\ref{p.Delta-preserves-order});
\item for $i=1,2$, $f^{s,\infty}_i$ is a dense subset of the (non-separated) one-dimensional manifold $f^{s}_i$ (Proposition~\ref{p.density-stable-manifold});
\item for $i=1,2$, each leaf $\ell\in f^{s}_i$ has a neighborhood $U_\ell$ in  $f^{s}_i$ so that the leaves in $U_\ell$ are totally ordered (Proposition~\ref{p.local-total-order}).
\end{itemize}
\end{proof}

%

Of course, the stable and the unstable direction play some symmetric roles, hence the same arguments as above alow to prove the following analog of Corollary~\ref{c.extension-s}:

\begin{corollary}
\label{c.extension-u}
The map $\Delta^u: f^{u,\infty}_1\longrightarrow f^{u,\infty}_2$ extends in a unique way to an order-preserving bijection $\widehat\Delta^u: f^{u}_1\longrightarrow f^{u}_2$.
\end{corollary}

\subsection{Mating $\widehat\Delta^s$ and $\widehat\Delta^u$: construction of the map $\widehat\Delta$} 

Now, we will mate the maps $\widehat\Delta^s$ and $\widehat\Delta^u$ to obtain a $\widehat\Delta:\widetilde S_1\to\widetilde S_2$. In view to that goal, we need the following lemma:

\begin{lemma}
Consider a leaf $\ell^s_1$ of the stable foliation $\widetilde\cF_1^s$ and a leaf $\ell^u_1$ of the unstable foliation $\widetilde\cF_1^u$. Then $\ell_1^s$ intersects $\ell_1^u$ if and only if $\widehat\Delta^s(\ell_1^s)$ intersects $\widehat\Delta^u(\ell_1^u)$.
\end{lemma}

\begin{proof}
The case where the leaves $\ell^s_1$ and $\ell^u_1$ belong to $f^{s,\infty}_1$ and $f^{u,\infty}_1$ is a consequence of Proposition~\ref{p.equivalence-transitions} (together with the definitions of the maps $\Delta^s$, $\Delta^u$ and $\Delta$): the leaves $\ell^s_1$ and $\ell^u_1$ intersect at $x$ if and only if the leaves $\widehat\Delta^s(\ell_1^s)=\Delta^s(\ell_1^s)$ and $\widehat\Delta^u(\ell^u_1)=\Delta^u(\ell^u_1)$ intersect at $\Delta(x)$. The general case follows by density of $f^{s,\infty}_i$ and $f^{u,\infty}_i$ in $f^{s,\infty}_i$ and $f^{u,\infty}_i$.
\end{proof}

Now we define a map $\widehat\Delta:\widetilde S_1\to\widetilde S_2$.  Let $\tilde x$ be any point in $\widetilde S_1$. Denote by $\ell_1^s$ (resp. $\ell_1^u$) the leaf of the stable foliation $\widetilde\cF^s_1$ (resp. the unstable foliation $\widetilde\cF^u_1$) passing through $x$. Recall that $x$ is the unique intersection point of $\ell^s_1$ and $\ell^u_1$. According to the preceding lemma, the stable leaf $\widehat\Delta^s(\ell^s_1)$ and the unstable leaf $\widehat\Delta^u(\ell^u_1)$ do intersect. According to Proposition~\ref{p.intersection-two-leaves}, the intersection is a single point. We define $\widehat\Delta(\tilde x)$ to be the unique intersection point of the leaves  $\widehat\Delta^s(\ell^s_1)$ and $\widehat\Delta^u(\ell^u_1)$. In other words, $\widehat\Delta$ is defined by 
\begin{equation}
\label{e.definition-extension}
\widehat\Delta(\ell^s_1\cap\ell^u_1)=\widehat\Delta^s(\ell^s_1)\cap\widehat\Delta^u(\ell^u_1).
\end{equation}
By construction, the map $\widehat\Delta$ is bijective, maps the foliations $\widetilde F_1^s,\widetilde F_1^u$ to the foliations $\widetilde F_2^s,\widetilde F_2^u$, preserving the orders on the leaf spaces. Since the leaf spaces are locally totally ordered (Proposition~\ref{p.local-total-order}), it follows that $\widehat\Delta$ is continuous. Hence $\widehat\Delta$ is a homeomorphism.

\begin{proposition}
\label{p.equivariance-Delta}
The map $\widehat\Delta:\widetilde S_1\to\widetilde S_2$ is equivariant with respect to the actions of the fundamental groups: for every $\gamma$ of $\pi_1(M_1)$, 
$$\widehat\Delta\circ\gamma=(\tilde\Phi_{in})_*(\gamma)\circ\widehat\Delta.$$
\end{proposition}

\begin{proof}
This is a rather immediate consequence of the construction of $\widehat\Delta$. First recall that $\widehat\Delta$ is a continuous extension of the map $\Delta:\widetilde S_1^\infty\to\widetilde S_2^\infty$ and recall that $\widetilde S_1^\infty,\widetilde S_2^\infty$ are dense subsets of $\widetilde S_1,\widetilde S_2$. As a consequence, it is enough to prove that $\Delta$ is equivariant with respect to the actions of the fundamental groups. Now recall that $\Delta$ is defined as the composition of three maps:
$$\Delta=(\chi_2)^{-1}\circ(\tilde\phi_{in})^{\otimes\ZZ}\circ\chi_1.$$ 
But we know that:
\begin{itemize}
\item the map $\chi_i$ commutes with the action of the fundamental group $\pi_1(M_i)$ for $i=1,2$ (Proposition~\ref{p.coding-commutes-pi-1}),
\item the map $\tilde\phi_{in}$ satisfies the following equivariance $\tilde\phi_{in}\circ\gamma=(\tilde\Phi_{in})_*(\gamma)\circ\tilde\phi_{in}$ (equation~\eqref{e.lift-phi-s-deck}).
\end{itemize}
This shows that the map $\Delta$ satisfies the equivariance relation $\Delta\circ\gamma=(\tilde\Phi_{in})_*(\gamma)\circ\Delta$, and completes the proof.
\end{proof}

\begin{proposition}
\label{p.conjugacy}
The map $\widehat\Delta:\widetilde S_1\to\widetilde S_2$ conjugates the Poincar\'e maps $\tilde\theta_1$ and $\tilde\theta_2$, \emph{that is} 
$$\widehat\Delta\circ\tilde\theta_1=\tilde\theta_2\circ\widehat\Delta.$$
\end{proposition}

\begin{proof}
On the one hand, for $i=1,2$, the coding map $\chi_i^s$ conjugates the Poincar\'e map $\tilde\theta_i$ on $\Wi S_i$ to the shift map on the symbolic space $\Sigma_i^s$ (Proposition~\ref{p.coding-conjugates-Poincare-shift}). On the other hand, the map $(\tilde\phi_{in})^{\otimes\ZZ_{\geq 0}}$ obviously conjugates the shift map on $\Sigma_1^s$ to  the shift map on $\Sigma_2^s$. Hence, $\Delta^s=(\chi^s_2)^{-1}\circ(\tilde\phi_{in})^{\otimes\ZZ_{\geq 0}}\circ\chi^s_1$ conjugates the action $\tilde\theta_1$ on $f^{s,\infty}_1$ to the action of $\tilde\theta_2$ on $f^{s,\infty}_2$. By density of $f^{s,\infty}_1$ in $f^{s}_i$, it follows that $\widehat\Delta^s$ conjugates the action $\tilde\theta_1$ on $f^{s}_1$ to the action of $\tilde\theta_2$ on $f^{s}_2$. Similalrly, $\widehat\Delta^u$ conjugates the action $\tilde\theta_1$ on $f^{u}_1$ to the action of $\tilde\theta_2$ on $f^{u}_2$. 
Finally, since $\Wi \Delta$ is defined by mating $\widehat\Delta^s$ and $\widehat\Delta^u$ (see~\eqref{e.definition-extension}), this implies that  $\Wi \Delta$ conjugates $\tilde\theta_1$  to $\tilde\theta_2$. 
\end{proof}

\subsection{From the map $\widehat\Delta$ to the orbital equivalence}

To conclude the proof of Theorem~\ref{t.main}, we need to introduce the orbit spaces of the Anosov flows $(Y_1^t)$ and $(Y_2^t)$. The \emph{orbit space} of $(Y_i^t)$ is by definition the quotient of the manifold $\widetilde M_i$ by the action of the flow $(Y_i^t)$. We denote it by $O_i$, and we denote by $\mathrm{pr}_i$ the natural projection of $\widetilde M_i$ on $O_i$. The action of the fundamental group $\pi_1(M_i)$ on $\widetilde M_i$ induces an action of this group on $O_i$. The two dimensional foliations $\widetilde\cF^s_i$ and $\widetilde\cF^u_i$ are leafwise invariant under the flow $(Y_i^t)$ and therefore can be projected in the orbit space $O_i$. They induce a pair $(g^s_i,g^u_i)$ of transverse 1-dimensional foliations on $O_i$. 

The orbit space $O_i$ by itself does not carry much information: indeed, $O_i$ is always separated manifold diffeomorphic to $\RR^2$ (see \cite[Proposition 2.1]{Fen} or \cite[Theorem 3.2]{Bar95}). The pair of transverse foliations $(g^s_i,g^u_i)$ carries a much more interesting information (see the work of Barbot and Fenley on the subject; good references are. Barbot's habilitation memoir~\cite{Bar06} and Barthelm\'e's lecture notes~\cite{Bar18}). The action of $\pi_1(M_i)$ on $O_i$ carries an even richer dynamical information: actually, this action characterizes the flow $(Y_i^t)$ up to topological equivalence (see Theorem~\ref{t.Barbot} below). 

Recall that $\Lambda$ denotes the maximal invariant set of the initial hyperbolic plug $(U,X)$, that $\Lambda_i$ denotes the projection of $\Lambda$ in the manifold $M_i=U/\psi_i$, and that $\widetilde\Lambda_i$ the complete lift of $\Lambda_i$ in the universal cover $\widetilde M_i$. Now, we denote by $L_i$ the projection of the set $\widetilde\Lambda_i$ in $O_i$. 

\begin{lemma}
\label{l.projection-surface}
The projection $\mathrm{pr}_i(\widetilde S_i)$ of the surface $\widetilde S_i$ in the orbit space $O_i$ is exactly the complement of the set $L_i$ in $O_i$.
\end{lemma}

\begin{proof}
The set $\Lambda$ is the union of the orbits of the vector field $X$ which remain in $U$ forever, \emph{i.e.} which do not intersect $\partial U$. Hence  the set $\Lambda_i=\pi_i(\Lambda)$ is the union of the orbits of the vector field $Y_i=(\pi_i)_*X$ which do not intersect the surface $S_i=\pi_i(\partial U)$. As a further consequence, $\widetilde\Lambda_i$ is the union of the orbits of the vector field $\widetilde Y_i$ which do not intersect the surface $\widetilde S_i$. This means that the projection of $\widetilde S_i$ in the orbit space $O_i$ is exactly the complement of the projection of the set $\widetilde \Lambda_i$.  
\end{proof}

Proposition~\ref{p.conjugacy} can be rephrased as follows: two points $x,x'\in\widetilde S_1$ belong to the same orbit of the flow $(\widetilde Y_1^t)$ if and only if the points $\widehat\Delta(x)$ and $\widehat\Delta(x')$ belong to the same orbit of the flow~$(\widetilde Y_2^t)$. As a consequence, the homeomorphism $\widehat\Delta:\widetilde S_1\to\widetilde S_2$ induces a homeomorphism 
$$\delta:\mathrm{pr}_1(\widetilde S_1)=O_1\setminus L_1\longrightarrow \mathrm{pr}_2(\widetilde S_2)=O_2\setminus L_2.$$
Since $\hat\Delta$ is equivariant with respect to the actions of the fundamental groups (Proposition~\ref{p.equivariance-Delta}), the homeomorphism $\delta$ is also equivariant: for every $\gamma\in\pi_1(M_1)$, 
$$\delta\circ\gamma=(\tilde\Phi_{in})_*(\gamma)\circ\delta.$$
Our next step is to extend the map $\eta$ to the whole orbit spaces.

\begin{proposition}
\label{p.extension-orbit-space}
The homeomorphism $\delta:O_1\setminus L_1\to O_2\setminus L_2$ can be extended in a unique way to a homeomorphism $\overline\delta:O_1\to O_2$, which is equivariant with respect to the actions of the fundamental groups of $M_1$ and $M_2$.  
\end{proposition}

We shall use the following general lemma of planar topology.

\begin{lemma}
\label{l.extension-orbit-space}
Let $A$ and $B$ be totally discontinuous subsets of $\RR^2$, and $h:\RR^2\setminus A\to \RR^2\setminus B$. Assume that, for every compact subset $K$ of $\RR^2$, the set $h(K\setminus A)$ is relatively compact in $\mathbb{R}^2$. Then $h$ can be extended to a homeomorphism of $\bar h:\RR^2\to\RR^2$.
\end{lemma}

This lemma is easy and certainly well-known by people working in planar topology, but we were not able to find it in the literature. We provide a proof for sake of completeness.

\begin{proof}
We proceed to the definition of $\bar h$. Let $x$ be a point in $A$. We pick a decreasing sequence $(X_n)_{n\geq 0}$ of compact connected subsets of $\RR^2$ so that $X_n\neq \{x\}$ for every $n$, and so that $\bigcap_n X_n=\{x\}$. For every $n\geq 0$, let $Y_n$ be the closure in $\RR^2$ of the set $h(X_n\setminus A)$. Our assumptions imply that $(Y_n)_{n\geq 0}$ is a decreasing sequence of non-empty compact connected subsets of $\RR^2$. As a consequence, the intersection $\bigcap_n Y_n$ must be a non-empty compact connected subset of $\RR^2$. Moreover, since $\bigcap_n X_n=\{x\}\subset A$, the intersection $\bigcap_n Y_n$ must be included in $B$. Since $B$ is totally disconnected, it follows that $\bigcap_n Y_n$ must be a singleton $\{y\}$. Standard arguments show that the point $y$ does not depend on the choice of the sequence $(X_n)$. We set $\overline h(x):=y$. Repeating the same procedure for each point $x\in A$, we get an extension $\bar h:\RR^2\to\RR^2$ of $h$. The continuity of $\bar h$ follows easily from its definition. 

Of course, the same procedure yields a continuous extension $\overline{h^{-1}}:\RR^2\to\RR^2$ of the map $h^{-1}:\RR^2\setminus B\to \RR^2\setminus A$. Since $\RR^2\setminus A$ and $\RR^2\setminus B$ are dense in $\RR^2$, the equalities $h\circ h^{-1}=\mathrm{Id}_{\RR^2\setminus B}$ and $h^{-1}\circ h=\mathrm{Id}_{\RR^2\setminus A}$ extend to $\bar h\circ \overline{h^{-1}}=\overline{h^{-1}}\circ \bar h=\mathrm{Id}_{\RR^2}$. This shows that that $\overline h$ is a homeomorphism and completes the proof.
\end{proof}

\begin{lemma}
\label{l.totally-discontinuous}
For $i=1,2$, the set $L_i$ is totally discontinuous in $O_i\simeq\RR^2$.
\end{lemma}

Let us introduce some terminology that will be used in the proof of Lemma~\ref{l.totally-discontinuous}. By a \emph{local section} of a vector field $Z$ on a $3$-manifold $P$, we mean a compact surface with boundary embedded in $P$ and transverse to $Z$. A $(Z^t)$-invariant set $\Omega\subset P$ is said to be \emph{transversally totally discontinuous} if $\Omega\cap\Sigma$ is totally discontinuous for every local section $\Sigma$ of $Z$. 

\begin{proof}
By our assumptions, the maximal invariant set $\Lambda_X$ of the hyperbolic plug $(U,X)$ contains neither attractors nor repellors. Since $\Lambda_X$ is a hyperbolic set, it follows that $\Lambda_X$ is transversally totally discontinuous. Hence the projection $\Lambda_i$ of $\Lambda_X$ in the manifold $M_i$ is also transversally totally discontinuous (recall that $\Lambda_X$ sits in the interior of $U$ and that the projection $p_i:U\to M_i$ is a homeomorphism in restriction to the interior of $U$). As a further consequence, the complete lift $\widetilde\Lambda_i$ of $\Lambda_i$ in the universal cover $\widetilde M_i$ is also transversally totally discontinuous. 

Now recall that $(\widetilde M_i,\widetilde Y_i)$ is topologically equivalent to $\RR^3$ equipped with the trivial vertical unit vector field. As a consequence, for every point $x\in\widetilde M_i$, we can find a local section $\Sigma$ of $\widetilde Y_i$, so that $x\in\Sigma$, and so that no orbit of $\widetilde Y_i$ intersects $\Sigma$ twice. This implies that the restriction to $\Sigma$ of the projection $\mathrm{pr}:\widetilde M_i\to O_i$ is one-to-one, hence a homeomorphism onto its image. Since $\widetilde\Lambda_i$ is transversally totally discontinuous, it follows that the set $L_i=\mathrm{pr}(\widetilde\Lambda_i)$ is totally discontinuous in $O_i$. 
\end{proof}

\begin{lemma}
\label{l.preserves-infinity}
For every compact set $K\subset O_1\simeq\RR^2$, the set $\eta(K\setminus L_1)$ has compact closure in $O_2\simeq\RR^2$.
\end{lemma}

\begin{proof}
For $i=1,2$, the surface $O_i\setminus  L_i$ has infinitely many ends. One of them is the end of $O_i\simeq\RR^2$, that we denote by $\infty_i$. The other ends are in one to one correspondance with the points of $L_i$ (since $L_i$ is totally discontinuous). Proving lemma~\ref{l.preserves-infinity} is equivalent to proving that the homeomophism $\eta:O_1\setminus L_1\to O_2\setminus L_2$ maps the end $\infty_1$ to the end $\infty_2$. 

From the viewpoint of the topology of the surface $O_i\setminus L_i$, nothing distinguishes $\infty_i$ from the other ends. Hence we need to introduce some dynamical invariants to prove that $\eta$ necessarily maps $\infty_1$ to $\infty_2$. 

For $i=1,2$, the foliation $\widetilde\cF^s_i$ induces a $1$-dimensional foliation $g^s_i$ on the  space $O_i$. We denote by $g^s_{i,0}$ the restriction of the foliation  $g^s_i$ to $O_i\setminus L_i$. According to Lemma~\ref{l.projection-surface}, $g^s_{i,0}$ can be obtained as the projection on $O_i$ of the foliation $\widetilde\cF^s_i\cap\widetilde S_i=\widetilde F^s_i$. As a consequence, $\eta$ maps the foliation $g^s_{1,0}$ to the foliation $g^s_{2,0}$. 

Since $O_i$ is a plane, every leaf of the foliation $g^s_i$ is a properly embedded line, going from $\infty_i$ to $\infty_i$ (recall that $\infty_i$ is the unique end of $O_i$). The leaves of $g^s_i=(\mathrm{pr}_i)_*\widetilde\cF^s_i$ that intersect $L_i=\mathrm{pr}_i(\Lambda_i)$ are the projections of the leaves of the lamination $W^s(\widetilde \Lambda_i)$. In particular, there exist leaves of $g^s_i$ that do not intersect $L_i$. As a consequence, there exist leaves of $g^s_{i,0}$ going from $\infty_i$ to $\infty_i$. On the other hand, if $x$ is an end of $O_i\setminus L_i$ corresponding to a point of $L_i$, then there does not exist any leaf of $g^s_{i,0}$ going from $x$ to $x$ (because every leaf $\ell$ of $g^s_{i,0}$ is a connected component of $\hat\ell\setminus L_i$ where $\hat\ell$ a line in $O_i$ going from $\infty_i$ to $\infty_i$). So, the foliation $g^s_{i,0}$ allows to distinguish $\infty_i$ from the other ends of $O_i\setminus L_i$. Since $\eta$ maps $g^s_{1,0}$ to $g^s_{2,0}$, it follows that $\eta$ must map $\infty_1$ to $\infty_2$. Since $\infty_i$ is the unique end of $O_i$, this exactly means that, for a compact set $K\subset O_1\simeq\RR^2$, the set $\eta(K\setminus L_1)$ has compact closure in $O_2\simeq\RR^2$.
\end{proof}

\begin{proof}[Proof of Proposition~\ref{p.extension-orbit-space}]
Lemmas~\ref{l.totally-discontinuous} and~\ref{l.preserves-infinity}, together with the fact that $O_1$ and $O_2$ are homeomorphic to $\RR^2$, show that we are exactly in the situation of Lemma~\ref{l.extension-orbit-space}. Applying this Lemma, we get a homeomophism $\bar\delta:O_1\to O_2$ extending $\eta$. The equivariance of $\bar\eta$ follows from those of $\delta$, by continuity and by density of $O_i\setminus L_i$ in $O_i$.
\end{proof}

We will now conclude the proof of Theorem~\ref{t.main} by using a result of Barbot. 

\begin{theorem}[See Theorem 3.4 of~\cite{Bar95}, or Proposition 1.36 and Corollaire 1.42 of~\cite{Bar06}]
\label{t.Barbot}
Two transitive Anosov flows are topologically equivalent if and only if there exist a homeomorphism between their orbit spaces, which is equivariant with respect to the actions of the fundamental groups, and which does not exchange the stable/unstable directions. 
\end{theorem}

\begin{proof}[Proof of Theorem~\ref{t.main}]
The Theorem is an immediate consequence of Proposition~\ref{p.extension-orbit-space} and Theorem~\ref{t.Barbot}.
\end{proof}

\vskip 1cm
\noindent Fran\c cois B\'eguin

\noindent {\small LAGA - UMR CNRS 7539 }

\noindent{\small Universit\'e Paris 13, 93430 Villetaneuse, FRANCE}

\noindent{\footnotesize{E-mail: beguin@math.univ-paris13.fr}}
\vskip 2mm

\noindent Bin Yu

\noindent {\small School of Mathematical Sciences}

\noindent{\small Tongji University, Shanghai 200092, CHINA}

\noindent{\footnotesize{E-mail: binyu1980@gmail.com }}

\end{document}